\documentclass[12pt]{article}

\usepackage{pdfsync}

\usepackage{latexsym,dsfont,amsfonts,amsmath,amssymb,theorem}
\usepackage{mathrsfs,mathcomp,stmaryrd, multirow}

\usepackage{pstricks,epsfig,psfrag}
\usepackage{url}

\graphicspath{{}{images/}{data/}}

\usepackage[hidelinks,hypertexnames=false,hyperindex=true,pdfpagelabels,linktoc=all]{hyperref}

\usepackage{enumitem}

\newtheorem{theorem}{Theorem}
\newtheorem{lemma}[theorem]{Lemma}

\newtheorem{corollary}[theorem]{Corollary}
\newtheorem{proposition}[theorem]{Proposition}

\newtheorem{algorithm}{Algorithm}

\newenvironment{proof}{\begin{trivlist}
    \item[\hskip\labelsep{\bf Proof.}]}{$\hfill\Box$\end{trivlist}}

{\theoremstyle{plain} \theorembodyfont{\rmfamily}
\newtheorem{remark}[theorem]{Remark}}
{\theoremstyle{plain} \theorembodyfont{\rmfamily}
\newtheorem{example}[theorem]{Example}}

\numberwithin{equation}{section}

\newcommand{\rhoext}{\rho^{\mathrm{ext}}}
\newcommand{\Rext}{R^{\mathrm{ext}}}

\newcommand{\Bext}{B^{\mathrm{ext}}}
\newcommand{\Qext}{Q^{\mathrm{ext}}}
\newcommand{\Lambdaext}{\Lambda^{\mathrm{ext}}}

\newcommand{\diam}{\operatorname{diam}}

\newcommand{\len}{{\ell}}
\newcommand{\wC}{\widetilde{C}}
\newcommand{\Zbar}{\overline{Z}}
\newcommand{\bZbar}{\overline{\bsZ}}

\newcommand{\bk}{{\boldsymbol{k}}}
\newcommand{\bsk}{{\boldsymbol{k}}}

\newcommand{\bsB}{\mathbf{B}}

\newcommand{\scA}{\mathscr{A}}
\newcommand{\IP}{\mathbb{P}}

\newcommand{\ha}{\widehat{a}}
\newcommand{\cN}{\mathcal{N}}
\newcommand{\bN}{\mathbb{N}}
\newcommand{\Mat}{Mat\'{e}rn\ }
\newcommand{\satop}[2]{\stackrel{\scriptstyle{#1}}{\scriptstyle{#2}}}

\newcommand{\bsDelta}{{\boldsymbol{\Delta}}}

\newcommand{\bsb}{{\boldsymbol{b}}}
\newcommand{\bsc}{{\boldsymbol{c}}}

\newcommand{\bsalpha}{{\boldsymbol{\alpha}}}
\newcommand{\bsgamma}{{\boldsymbol{\gamma}}}

\newcommand{\bsnu}{{\boldsymbol{\nu}}}
\newcommand{\bsm}{{\boldsymbol{m}}}
\newcommand{\bsr}{{\boldsymbol{r}}}

\newcommand{\bsp}{{\boldsymbol{p}}}

\newcommand{\bst}{{\boldsymbol{t}}}

\newcommand{\bsv}{{\boldsymbol{v}}}

\newcommand{\bsx}{{\boldsymbol{x}}}

\newcommand{\bsy}{{\boldsymbol{y}}}
\newcommand{\bsY}{\boldsymbol{Y}}
\newcommand{\bsz}{{\boldsymbol{z}}}
\newcommand{\bsZ}{{\boldsymbol{Z}}}
\newcommand{\bsw}{{\boldsymbol{w}}}

\newcommand{\bszero}{{\boldsymbol{0}}}

\newcommand{\rd}{\mathrm{d}}
\newcommand{\ri}{\mathrm{i}}
\newcommand{\bbR}{\mathbb{R}}
\newcommand{\bbZ}{\mathbb{Z}}
\newcommand{\bbP}{\mathbb{P}}
\newcommand{\bE}{\mathbb{E}}
\newcommand{\bbE}{\mathbb{E}}
\newcommand{\bbN}{\mathbb{N}}

\newcommand{\calF}{\mathcal{F}}
\newcommand{\cI}{\mathcal{I}}
\newcommand{\cA}{\mathcal{A}}
\newcommand{\cT}{\mathcal{T}}

\newcommand{\cF}{\mathcal{F}}
\newcommand{\calG}{\mathcal{G}}
\newcommand{\cG}{\mathcal{G}}
\newcommand{\calO}{\mathcal{O}}

\renewcommand{\Re}{\mathfrak{Re}}
\renewcommand{\Im}{\mathfrak{Im}}

\newcommand{\tr}{{\sf T}}

\newcommand{\setu}{\mathfrak{u}}
\newcommand{\mask}[1]{}

\newcommand{\tu}{{u}}

\newcommand{\bR}{\mathbb{R}}

\makeatletter\@addtoreset{equation}{section}\makeatother

\definecolor{darkred}{RGB}{139,0,0}
\definecolor{darkgreen}{RGB}{0,100,0}
\definecolor{darkmagenta}{RGB}{139,0,139}
\definecolor{darkpurple}{RGB}{110,0,180}
\definecolor{darkblue}{RGB}{40,0,200}
\definecolor{darkorange}{RGB}{255,140,0}

\newcommand{\rcov}{r_{\mathrm{cov}}}

\begin{document}
\date{\today}

\title{Circulant embedding with QMC -- analysis for elliptic PDE with lognormal coefficients}
\author{I. G. Graham${}^{1}$, F. Y. Kuo${}^2$, D. Nuyens${}^3$, R. Scheichl${}^1$, and
  I. H. Sloan${}^2$ }

\maketitle

\begin{center}
\begin{scriptsize}

\vspace{-0.75cm}

\noindent
${}^1$ Dept of Mathematical Sciences, University of Bath, Bath BA2
7AY UK \\\indent {\tt I.G.Graham@bath.ac.uk}, \ {\tt
  R.Scheichl@bath.ac.uk}

\vspace{0.1cm}

\noindent
${}^2$ School of Mathematics and Statistics, University of NSW,
Sydney NSW 2052, Australia. \\ \indent {\tt f.kuo@unsw.edu.au},\  {\tt
  i.sloan@unsw.edu.au}

\vspace{0.1cm}

\noindent ${}^3$ Department of Computer Science, KU Leuven,
Celestijnenlaan 200A, B-3001 Leuven.
\\ \indent {\tt dirk.nuyens@cs.kuleuven.be}

\end{scriptsize}
\end{center}

\begin{abstract}
In a previous paper (J.\ Comp.\ Phys.\ \textbf{230} (2011), 3668--3694,
see ref.~\cite{GrKuNuScSl:11}), the authors proposed a new practical
method for computing expected values of functionals of solutions for
certain classes of elliptic partial differential equations with random
coefficients. This method was based on combining quasi-Monte Carlo
(QMC) methods for computing the expected values with circulant embedding
methods for sampling the random field on a regular grid.
It was found capable of
handling fluid flow problems in random heterogeneous media with high
stochastic dimension, but \cite{GrKuNuScSl:11} did not provide a
convergence theory. This paper provides a convergence analysis for the method
in the case when the QMC method is a specially designed randomly shifted
lattice rule.
The convergence result depends on the eigenvalues of the underlying
nested block circulant matrix and can be independent of the number of
stochastic variables under certain assumptions.
In fact the QMC analysis applies to general factorisations of the covariance
matrix to sample the random field.
The error analysis for the underlying fully discrete finite element method
allows for locally refined meshes (via interpolation from a regular sampling
grid of the random field). Numerical results on a non-regular domain with
corner singularities in two spatial dimensions and on a regular domain in three
spatial dimensions are included.
\end{abstract}

\begin{footnotesize}
\end{footnotesize}

\section{Introduction} \label{sec:Intro}

In the paper \cite{GrKuNuScSl:11}, the present authors proposed a new
practical algorithm for solving a class of elliptic partial differential
equations
 with coefficients given by statistically homogeneous lognormal
random fields -- and in particular for computing expected values of
spatial functionals of such solutions. In this algorithm, the required
expected value is written as a multidimensional integral of (possibly)
high dimension, which is then approximated by a quasi-Monte Carlo (QMC)
method. Each evaluation of the integrand is obtained by using a fully
discrete finite element (FE) method  to approximate the PDE. A key
original feature of the method in \cite{GrKuNuScSl:11} was the procedure
for sampling the random field: instead of sampling the continuous random
field by a truncated Karhunen--Lo\`{e}ve (KL) expansion, the field was
sampled discretely on a regular grid covering the domain and then
interpolated at the (irregularly spaced) quadrature points. This
completely eliminated  the problem of truncation error from the KL
expansion, but requires the factorisation of a dense matrix of dimension
equal to the number of sample points.  In \cite{GrKuNuScSl:11} this was
done using a circulant embedding technique. The method was found to be
effective even for problems with high stochastic dimension, but
\cite{GrKuNuScSl:11} did not contain a convergence analysis of the
algorithm.

The main purpose of the present paper is to provide an analysis for a
method closely related to that of \cite{GrKuNuScSl:11}, with an error
bound that is independent of stochastic dimension, and a convergence rate
faster than that of a simple Monte Carlo method. The setting differs in
two ways from \cite{GrKuNuScSl:11}: first, the FE method considered here
is the standard nodal  FE method for elliptic problems, whereas in
\cite{GrKuNuScSl:11}  the mixed FE method  was used; and second, the QMC
method considered here is a specially designed randomly shifted lattice
rule (see \eqref{eq:QMC} below), instead of using Sobol$'$ points
as in \cite{GrKuNuScSl:11}.  (We expect the present analysis can be
extended to mixed FEs using results in \cite{GrScUl:13}, but
do not attempt this here.)

Thus our PDE model (written initially in strong form) is
\begin{align} \label{eq:diffeq}
  -\nabla \cdot (a(\bsx,\omega) \nabla u(\bsx,\omega))
  &\,=\, f(\bsx)\,
   \qquad \text{for } \bsx \in D\subseteq [0,1]^d,\ \ \text{and almost all}\ \
   \omega \in \Omega .
\end{align}
Given
a functional $\cG$ of $u$ with respect to the spatial variable $\bsx$,
our aim here (as in \cite{GrKuNuScSl:11}) is to compute efficiently and
accurately $\bE[\cG(u)]$, the expected value of $\cG(u(\cdot,
\omega))$.
The (spatial) domain $D \subset \mathbb{R}^d$ ($d = 1,2,3$) in \eqref{eq:diffeq} is
assumed to be a bounded interval ($d=1$), polygon ($d=2$) or Lipschitz
polyhedron ($d = 3$), while $\Omega$ is the set of events in a suitable
probability space $(\Omega,\cA,\bbP)$. The solution $u$ is required to
satisfy the homogeneous Dirichlet condition $u = 0$ on the boundary
$\partial D$ of $D$. The spatial domain $D$ is allowed to be irregular but
we assume for convenience that it can be embedded in the $d$-dimensional
unit cube; this is always possible after a suitable affine scaling. (The
length-scale of our random field is therefore always considered with
respect to the unit cube.) The driving term $f$ is for simplicity taken to
be deterministic.

We consider the lognormal case where
\begin{equation}\label{eq:field}
 a(\bsx, \omega) \,=\, \exp (Z(\bsx,\omega)),
\end{equation}
with $Z(\bsx, \omega)$ a Gaussian random field with prescribed mean
$\Zbar(\bsx)$ and covariance
\begin{equation}\label{eq:covar}
  \rcov(\bsx,\bsx') \,:=\,
  \mathbb{E}[(Z(\bsx,\cdot)-\Zbar(\bsx))(Z(\bsx',\cdot)-\Zbar(\bsx')],
\end{equation}
where the expectation is with respect to the Gaussian measure. Lognormal
random fields are commonly used in applications, for example in hydrology
(see, e.g., \cite{NaHaSu:98a} and the references there).
Throughout
we will assume that $Z$ is \emph{stationary} (see, e.g.,
\cite[p.~24]{Ad:81}), i.e., its covariance function satisfies
\begin{equation} \label{eq:defrho1}
 \rcov(\bsx,\bsx') \,=\, \rho(\bsx-\bsx').
\end{equation}
Strictly speaking, $\rho$ only needs to be defined on a sufficiently large
ball $B(0,\diam(D))$ for the prescription above, but as in many
applications we assume that it is defined on all of $\bR^d$. A particular
case that will be discussed extensively, is the \Mat covariance, where
$\rho$ is isotropic, i.e., $\rho$ depends on $\bsx$ only through its
Euclidean length $\|\bsx\|_2$.

In the present paper \eqref{eq:diffeq} is discretised by piecewise linear
finite elements in space, using simplicial meshes with maximum diameter
$h$, and a simple low order quadrature rule for suitable approximation of
the associated stiffness matrix. In consequence, the only values of the
stochastic coefficient $Z(\bsx, \omega)$ that enter the FE computation are
its values at the quadrature points. However, the FE quadrature points
will in general be irregularly distributed, and (for refined meshes) very
large in number, typically rendering a direct evaluation of the field at
the quadrature points prohibitively expensive. It is much more efficient,
as explained below, to instead evaluate exactly the realisation of the
field at a uniform grid of
\[
 M \,=\, (m_0+1)^d
\]
points $\bsx_1,\bsx_2,\ldots,\bsx_M$ on the $d$-dimensional unit cube
$[0,1]^d$ (containing the domain
$D$), with a fixed integer $m_0$ and with grid spacing $h_0 := 1/m_0$. We
assume further that $h_0\sim h$. (The extension to general tensor product
grids with different mesh sizes in the different coordinate directions is
straightforward and not discussed here.) We use \emph{multilinear interpolation}
to obtain a sufficiently good approximation
of the field $Z(\bsx, \omega)$ at any other spatial point $\bsx\in D$,
i.e., we use repeated linear
interpolation in each coordinate direction with respect to the vertices of
the surrounding grid cell.

At this stage of the algorithm, the output is the approximate FE solution
$u_h(\bsx, \omega)$, which inherits randomness from the input data
\begin{equation*}
 \bsZ(\omega) \,:=\, (Z(\bsx_1,\omega), \ldots , Z(\bsx_M, \omega))^{\top} \ .
\end{equation*}
The $M$-vector $\bsZ(\omega)$ is a Gaussian random vector with a
covariance structure inherited from the continuous field $Z$. Thus it has
mean $\overline{\bsZ} :=
(\overline{Z}(\bsx_1),\ldots,\overline{Z}(\bsx_M))^\top$ and a positive
definite covariance matrix
\begin{equation} \label{eq:Rmatrix}
  R \,=\,  [\rho(\bsx_i-\bsx_j)]_{i,j=1}^M.
\end{equation}
Because of its finite length, $\bsZ(\omega)$ can be expressed exactly (but
not uniquely) as a linear combination of a finite number of i.i.d.\
standard normal random variables, i.e., as
\begin{equation}\label{eq:Gauss_exp1}
 \bsZ(\omega) \,=\,  B \bsY(\omega)\  + \ \bZbar \  ,
 \quad \text{where} \quad  \bsY \sim \cN(\boldsymbol{0}, I_{s \times s}).
\end{equation}
for some real $M\times s$ matrix $B$ with $s\ge M$ satisfying
\begin{equation} \label{eq:Rfactor}
  R \,=\, BB^\top.
\end{equation}
To see this, note that \eqref{eq:Gauss_exp1} and \eqref{eq:Rfactor} imply
$$
  \mathbb{E}[(\bsZ-\bZbar)(\bsZ-\bZbar)^{\top}]
  \,=\, \mathbb{E}[B\bsY \bsY^\top B^{\top}]
  \,=\, B\,\mathbb{E}[\bsY \bsY^\top]B^{\top}=BB^{\top} \,=\, R.
$$

An efficient computation of a suitable factorisation \eqref{eq:Rfactor}
using the extension of $R$ to a nested block circulant matrix and then
diagonalisation using FFT (the ``circulant embedding method'') is
described in detail in \cite{ChWo:94,ChWo:97,DiNe:97,GrKuNuScSl:11}.
 It is essential for that approach that the random field is sampled
  on a uniform grid of points.
In a related paper \cite{paper1}, we have analysed certain key
properties of the circulant extension and its factorisation, which will be
crucial for efficiency and for the dimension independence of the QMC
convergence analysis in this paper. Other
approaches, such as Cholesky factorisation or direct spectral
decomposition, could also be used to find a factorisation of the form
\eqref{eq:Rfactor}. These alternative approaches have the advantage that
they do not require the sample grid to be uniform, but when $M$ is
large these approaches are likely to be prohibitively expensive. Some
ideas of how to overcome this problem using a
pivoted Cholesky factorisation or hierarchical matrices can be found
in \cite{HaPeSc:12} or \cite{BlCoDa:16,FeKuSl:17},
respectively.

From now on, realisations of the random vector $\bsY(\omega)$ are
denoted by $\bsy$. Thus, $\bsy$ contains $s$ independent realisations of
$\cN(0,1)$. Hence, if
$F: \bR^s \rightarrow \bR$ is any Lebesgue measurable function then the
expected value of  $ F(\bsY(\omega))$ may be written as
\begin{equation}\label{eq:unit_cube}  I_s(F) \ := \int_{\bR^s} F(\bsy)
\prod_{j=1}^s \phi(y_j)  \, \rd \bsy   \  = \  \int_{{(0,1)^s}}
 F(\Phi_s^{-1}(\bsv)) \, \rd \bsv \ ,
\end{equation}
where $\phi$ is the one-dimensional standard normal probability density,
and  $\Phi_s^{-1}$ is the inverse of the cumulative normal
distribution function applied componentwise on $(0,1)^s$.
Since $u_h(\cdot, \omega)$ is derived from $\bsY(\omega)$, we make the
notational convention
\begin{equation}\label{eq:convention}
u_h(\bsx, \omega) \ =  \ u_h(\bsx,\bsy),
\end{equation} and so our
approximation to $\bE[\cG(u)]$ is
\begin{equation}
\bE[\cG(u_h)] \ = \ I_s(F), \quad \text{with}\quad  F(\bsy) = \cG(u_h(\cdot, \bsy)) .
\label{eq:int0}
\end{equation}

Because $s\geq M$, the integral \eqref{eq:unit_cube} can have very high
dimension. However, two important empirical findings in \cite{GrKuNuScSl:11}
were that (for all the applications considered) the accuracy of the QMC cubature rule
based on Sobol$'$ points did not appear to be affected by the size
of~$s$ and that it was always superior to classical Monte Carlo (MC) methods.
Successful computations with $s \sim 4\times 10^6$ were reported in
\cite{GrKuNuScSl:11}. One aim of the present work is to provide a
rigorous proof of the independence of the QMC error on the
dimension~$s$ (under appropriate conditions) for a specially designed
randomly shifted lattice rule. Furthermore, we will also prove here
the superior asymptotic convergence rate of QMC over MC in this setting.

The accuracy of the approximation to $\bE[\cG(u)]$ depends on the FE mesh
diameter~$h$ (through the FE error and the interpolation error),
as well as on the
number $n$ of QMC points. We analyse convergence with respect to both
these parameters. Our first set of theoretical results concern the
accuracy with respect to $h$. In particular, for bounded linear
functionals $\cG$ (with respect to the spatial variable $\bsx$) and under
suitable assumptions, one result,  obtained in \S \ref{subsec:Error},  is
that
\begin{equation}\label{eq:functionals}
  \vert \,\bE \left[\cG(u) - \cG(\tu_h)\right] \vert  \ \leq  \ C h^{2t}\ ,
\end{equation}
for some parameter $t \in (0,1]$, determined by the smoothness of
realisations of $a(\bsx, \omega)$, and with a constant $C$
independent of $h$. This result differs from that of
\cite{TSGU13} through the use of interpolation to approximate the
random vector $\bsZ$.

Then, a substantial part of the paper is concerned with the convergence of
the QMC rules for \eqref{eq:unit_cube}. In particular, we consider
randomly shifted lattice rules,
\begin{equation}\label{eq:QMC}
  Q_{s,n}(\bsDelta, F)
  \, := \,
  \frac{1}{n} \sum_{k=1}^n F  \left( \Phi^{-1}_s (\bsv_k) \right)
\;, \quad \text{with} \quad   \bsv_k = {\rm frac} \left( \frac{k\,\bsz}{n} + \bsDelta \right),
\end{equation}
where $\bsz\in \mathbb{N}^s$ is some suitably chosen generating vector,
$\bsDelta \in [0,1]^s$ is a  uniformly distributed random shift, and
``${\rm frac}$'' denotes taking the fractional part of every
component in a vector. For the particular integrand $F(\bsy) :=
\cG(\tu_h(\cdot,\bsy))$, we provide in \S \ref{sec:qmc} upper bounds on
the error of approximating the integral \eqref{eq:int0} by \eqref{eq:QMC}.
In particular, in Theorem \ref{thm:QMC}, we give sufficient conditions on the
matrix $B$ appearing in \eqref{eq:Rfactor} for the
root-mean-square error to satisfy an estimate of the form:
\begin{equation}
\label{def:error}
\sqrt{\bbE_\bsDelta
\left[
|\, I_s(F) - Q_{s,n}(\bsDelta, F) |^2
\right]}  \ \leq \ \   C_\delta\,  n^{-(1-\delta)}\ ,
\end{equation}
for arbitrary $\delta > 0$. Here, $\bbE_\bsDelta$ denotes expectation with
respect to the random shift~$\bsDelta$. Moreover, we also provide
conditions under which the constant $C_\delta$ in \eqref{def:error} is
independent of $s$. Our proof of Theorem~\ref{thm:QMC} differs from the
corresponding result in~\cite{GKNSSS:15} because of the use of multilinear
interpolation of the random field in space, and because of the different
meaning of the parameters $y_j$. The problem is no longer about truncating
an infinite Karhunen--Lo\`eve expansion, but rather of dealing with a
sequence of matrix problems with ever increasing size $s$.

Finally, combining \eqref{eq:functionals} and \eqref{def:error}, the
overall error estimate is
\begin{equation*}
\sqrt{ \bbE_\bsDelta
\left[
|\,\bbE[\cG(u)] - Q_{s,n}(\bsDelta, F)|^2
\right]}  \,\le\, \sqrt{2}\, {C} \, h^{2t}  +  \sqrt{2}\,C_\delta\, n^{-(1-\delta)} .
\end{equation*}
Although the algorithm in \cite{GrKuNuScSl:11} applies to both linear and
nonlinear functionals, our theory at present is restricted to the linear
case. In the numerical experiments, we will use the average of $q$
different random shifts as our final QMC estimator, bringing the total
amount of integrand evaluations (i.e., PDE solves) to $N=q\,n$. We compare
with a classical Monte Carlo (MC) method for which we use $N$
i.i.d.\ random samples $\bsw_k \sim U[0,1)^s$,  i.e.,
\begin{equation}\label{eq:MC}
  Q^{\textrm{MC}}_{s,N}(F)
  \, := \,
  \frac{1}{N} \sum_{k=1}^N F  \left( \Phi^{-1}_s (\bsw_k) \right)
\;, \quad \text{with} \quad   \bsw_k \sim U[0,1)^s .
\end{equation}

The layout of the paper is as follows. The PDE with random coefficient and
its FE approximation with quadrature are described in
\S\ref{subsec:Model}. The estimate \eqref{eq:functionals} is proved in
\S\ref{subsec:Error}. The QMC theory is given in
\S\ref{sec:qmc}. In particular, one of the key results proved in
\S\ref{subsec:errorest} is the upper bound~\eqref{def:error}.
A sufficient condition on $B$ for this result in the case of circulant
embedding is identified in \S\ref{sec:expect}. The circulant embedding
algorithm is summarised briefly in \S\ref{sec:expect} and we refer to
\cite{paper1} for its theoretical analysis.  Numerical experiments are
given in \S\ref{sec:Numerical}, illustrating the performance of the
algorithm on PDE problems on an irregular domain with corners and holes
in two space dimensions, as well as on the unit cube in three
dimensions.


\section{Finite element implementation and analysis}\label{sec:flow}

In \S\ref{subsec:Model}, we first give the algorithmic details of our
practical finite element method, before proving error estimates for this
method in \S\ref{subsec:Error}, in particular Theorem~\ref{thm:H1error}
and Corollary~\ref{cor:functionals}.

\subsection{Model formulation and implementation}  \label{subsec:Model}

We start with  \eqref{eq:diffeq} written pathwise in weak form: seek
$u(\cdot,\omega) \in V := H^1_0(D)$ such that
\begin{equation}\label{eq:weak_form}
\mathscr{A}(\omega; u(\cdot,\omega), v) \ = \   \langle f, v\rangle\  \quad \text{for all} \ v
\in V\  \text{and for almost all} \ \omega \in \Omega\ ,
\end{equation}
where
$$
\scA(\omega; w,v) \,:=\,
\int_D a(\bsx, \omega)\, \nabla w(\bsx)\cdot \nabla v(\bsx)\,\rd\bsx\ , \quad
w, v \in V  ,
$$
and $a(\bsx, \omega)$ is given by \eqref{eq:field} --
\eqref{eq:defrho1}. The norm in $V$ is
$\|v\|_V := \|\nabla v\|_{L^2(D)}$. For simplicity we assume that $ f \in L^2(D)$,
  so that $\langle f,v\rangle$ reduces to the
$L^2(D)$ inner product in
\eqref{eq:weak_form}. In general, it denotes the duality pairing
between $V$ and its dual space $V' := H^{-1}(D)$.

To discretise \eqref{eq:weak_form}  in the physical domain $D$,  let
$\{\mathcal{T}_h\}_{h>0}$ denote a family of conforming, simplicial meshes
on $D$, parametrised by the maximum mesh diameter $h := \max_{\tau \in
\mathcal{T}_h} \diam(\tau)$ with $\diam(\tau) := \max_{\bsx,\bsx' \in
\tau} \|\bsx - \bsx'\|_2$\,. On this mesh, we let $V_h\subset V$
denote the usual finite element space of continuous piecewise linear
functions that vanish on $\partial D$. We assume that
$\text{dim}(V_h) = \mathcal{O}(h^{-d})$. This
includes many locally refined mesh families, including anisotropic
refinement in 3D (e.g., \cite{Apel:1999}, \cite{BaNiZi:07}). Since any
function in $V_h$ has a piecewise constant gradient, we have
\begin{equation}\label{eq:AonVh}
\scA(\omega; w_h,v_h) \,=\,
\sum_{\tau \in \cT_h} a_\tau (\omega)\, (\nabla w_h(\bsx)\cdot \nabla v_h(\bsx)) \bigr\vert_\tau  \quad
\mbox{for all } v_h,  w_h \in V_h\ ,
\end{equation}
where
\begin{equation}\label{eq:exactint}
  a_\tau(\omega) \ := \
  \int_\tau a(\bsx, \omega) \,\rd \bsx, \quad \tau\in \mathcal{T}_h\ .
\end{equation}

We approximate the required integrals \eqref{eq:exactint} using a
two-stage interpolation/quadrature process as follows. Recall the uniform
grid on the cube containing $D$, with points $\bsx_j$ for $j = 1, \ldots,
M$, and grid spacing $h_0\sim h$, defined in \S\ref{sec:Intro}. Let $\bsx
\in \overline{D}$ and let $\{\bst_{i,\bsx}\}_{i=1}^{2^d} \subset
\{\bsx_1,\ldots,\bsx_M\}$ be the vertices of the surrounding grid cell
labelled in arbitrary order. Since multilinear interpolation is done by
repeatedly applying linear interpolation in each coordinate direction, we
can write the interpolated value of $g$ at $\bsx$ as a convex combination
of the surrounding vertex values $\{\bst_{i,\bsx}\}_{i=1}^{2^d}$, i.e.,
\begin{equation}\label{eq:multilin}
  \cI_{h_0}(g;\{\bsx_j\}_{j=1}^M)(\bsx)
  =
  \sum_{i=1}^{2^d} w_{i,\bsx} \, g(\bst_{i,\bsx})
  ,\qquad
  \text{with}\quad
  \sum_{i=1}^{2^d} w_{i,\bsx} = 1
  \;\text{ and }\;
  0 \le w_{i,\bsx} \le 1.
\end{equation}
The operator $\cI_{h_0}:C(\overline{D})\to C(\overline{D})$ is linear and
satisfies $\cI_{h_0}(g;\{\bsx_j\}_{j=1}^M)(\bsx_i) = g(\bsx_i)$, for every
point $\bsx_i$ of the uniform grid.

Let us further define an $r$-point quadrature rule on each element $\tau$,
which is exact for constant functions, has positive weights $\mu_{\tau,k}
\ge 0$ and only uses quadrature points $\bsx_{\tau,k} \in \tau$, i.e.,
\begin{equation} \label{eq:quad}
  Q_\tau(g) :=
  \sum_{k=1}^r \mu_{{\tau},{k}} \, g(\bsx_{{\tau},{k}}),
  \quad \text{with} \ \
  \sum_{k=1}^r \mu_{{\tau},{k}} = \vert \tau \vert \ \text{ and } \ \mu_{\tau,k} \ge 0.
\end{equation}
Here, $|\tau|$ denotes the volume of $\tau$. The quadrature points
$\bsx_{{\tau},{k}}$ in \eqref{eq:quad} are unrelated to the uniform grid
points $\bsx_j$ in general. Examples of rules satisfying \eqref{eq:quad}
are the centroid, nodal or face-centroid rules.

Using the rule \eqref{eq:quad} to approximate all the integrals
\eqref{eq:exactint} would require evaluating $a(\cdot, \omega)$ at the
union of all the  (in general irregularly distributed) quadrature points
$\{ \bsx_{{\tau},{k}} \} $, which could be costly. We avoid that and
compute the field only at the points of the uniform grid. We then
interpolate these values using
$\cI_{h_0}(a(\cdot,\omega);\{\bsx_j\}_{j=1}^M)$ and then approximate
$a_\tau(\omega)$ using \eqref{eq:quad}. In summary, we approximate the
bilinear form in \eqref{eq:AonVh} by
\begin{equation}\label{eq:defAh}
\scA_h(\omega;w_h,v_h)
\,:=\,
\sum_{\tau\in\mathcal{T}_h}  \widehat{a}_\tau(\omega)  \,
(\nabla w_h \cdot \nabla v_h) \bigr\vert_\tau \ ,
\end{equation}
where
\begin{equation} \label{eq:defatau}
\widehat{a}_{\tau}(\omega)
 \,:=\,
Q_\tau\bigl(\cI_{h_0}(a(\cdot,\omega);\{\bsx_j\}_{j=1}^M)\bigr)
= (Q_\tau \circ \cI_{h_0})(a(\cdot,\omega))
.
\end{equation}

\begin{proposition} \label{prop1} For all $\tau \in \cT_h$, there is a sparse positive vector
$\bsp_\tau =(p_{\tau,1},\cdots,p_{\tau,M}) \in \mathbb{R}^M$ such that
$$
\widehat{a}_\tau(\omega)
\ = \
\sum_{j=1}^M p_{\tau,j} \, a(\bsx_j,\omega)
, \quad \text{and} \quad
  \ha_\tau(\omega)  \ \geq \ \vert \tau \vert\, {a}_{\min,M}(\omega)\,, \quad
  \text{for all} \ \ \tau \in
\cT_h\,,
$$
where $a_{\min,M}(\omega) \,:=\, \min_{1\le j\le M} a(\bsx_j,\omega)$.
\end{proposition}
\begin{proof}
It follows from~\eqref{eq:defatau} together with~\eqref{eq:multilin}
and~\eqref{eq:quad} that
\begin{equation}\label{eq:approxhata}
 \widehat{a}_\tau(\omega) \,= \,
 \sum_{k=1}^r \sum_{i=1}^{2^d} \mu_{\tau,k} \,  w_{i,\bsx_{\tau,k}}
 a(\bst_{i,\bsx_{\tau,k}},\omega).
\end{equation}
The second result then follows from the definition of $a_{\min,M}(\omega)$
and the fact that the coefficients $\mu_{\tau,k}$ and
$w_{i,\bsx_{\tau,k}}$ are all positive and their sum is $|\tau|$.
\end{proof}

Extending the notational convention \eqref{eq:convention}, we may thus write
our discrete finite element method for \eqref{eq:weak_form} as the problem of
finding $u_h(\cdot, \bsy)$ which satisfies
\begin{equation}\label{fepar}
\scA_h(\bsy;u_h(\cdot,\bsy),v_h)
\ = \
\langle f,v_h\rangle,
\qquad \mathrm{for}\
\mathrm{all} \quad v_h \in V_h\,  \quad \bsy \in \mathbb{R}^s,
\end{equation}
where $\scA_h(\bsy;w_h,v_h)$ is identified with $\scA_h(\omega;w_h,v_h)$.

\subsection{Finite element error analysis}
\label{subsec:Error}

Let us first define some relevant function spaces. Let
$C^1(\overline{D})$ denote the space of continuously
differentiable functions on $D$ with seminorm $\vert \phi
\vert_{C^1(\overline{D})} := \sup_{\bsx \in
\overline{D}} |\nabla \phi(\bsx)|$. For $\beta \in
(0,1)$, let  $C^{\beta}({\overline{D}})$ denote the space of
H\"{o}lder continuous functions  on $D$ with exponent
$\beta$ and let
$\vert \phi \vert_{C^\beta(\overline{D})} := \sup_{\bsx_1,\bsx_2 \in
\overline{D}\,:\,\bsx_1 \not= \bsx_2}
 |\phi(\bsx_1) - \phi(\bsx_2)|/\|\bsx_1 - \bsx_2\|_2^\beta <\infty$
denote the \emph{H\"older coefficient} which is, in fact, a seminorm.
Also let $L^p(\Omega,X)$ denote the space of
all random fields in a Banach space $X$ with bounded $p$th moments over
$\Omega$.

We assume throughout that  $Z(\cdot,\omega) \in C^\beta(\overline{D})$,
for some $\beta \in (0,1]$, $\IP$-almost surely. Since $Z(\bsx,\omega)$ is
Gaussian, it follows from Fernique's Theorem that
$\|a\|_{L^p(\Omega,C^\beta(\overline{D}))}$ is finite, for all $p \in
[1,\infty)$ (see \cite{ChScTe:11}). Moreover, this implies that
$\Vert a_{\max} \Vert_{L^p(\Omega)} < \infty$  and
$\|1/a_{\min}\|_{L^p(\Omega)} < \infty$, for all $p \in [1,\infty)$, where
$a_{\min}(\omega) \,:=\, \min_{\bsx \in \overline{D}} a(\bsx,\omega)$ and
$a_{\max}(\omega) \,:=\, \max_{\bsx \in \overline{D}} a(\bsx,\omega)$.

Models where realisations of $a(\bsx,\omega)$ lack smoothness are often of
interest in applications, and a  class of coefficients of particular
significance is given by the \Mat class with smoothness parameter $\nu
\geq  1/2$, described in detail in Example~\ref{ex:Mat}. For $\nu \le 1$,
realisations are in $C^{\beta}({\overline{D}})$ $\IP$-almost surely, for
all $0<\beta<\nu$ (see, e.g., \cite{LPS14,GKNSSS:15}).

There are two factors that limit the convergence rate of the finite
element error: (i) the regularity of the coefficient field
$a(\cdot,\omega)$ and (ii) the shape of the domain $D$. Since
$a(\cdot,\omega) \in C^\beta(\overline{D})$, then (if $\partial D$ is
smooth enough),  we have $u(\cdot, \omega) \in H^{1+t}(D)$ for all $0\leq
t \leq \beta$.  Here, when $\beta < 1$,  the loss of $H^2$ regularity is
global and the resulting reduction in the finite element convergence rate
cannot be corrected by local mesh refinement.
On the other hand, the influence of corner or edge singularities can
typically be eliminated by suitable local mesh refinement near $\partial
D$.

Using the notation in \cite[Def.~2.1]{TSGU13}, let
$\lambda_\Delta(D)$ be the order of the strongest singularity
of the Dirichlet-Laplacian on $D$.
Then $u(\cdot,\omega) \in H^{1+t}(D)$, for all $t \le \lambda_\Delta(D)$
and $t < \beta$, and $\|u\|_{L_p(\Omega,H^{1+t}(D))}$ is bounded for all
$p \in [1,\infty)$ (see \cite[Lem.~5.2]{TSGU13}).
When $\lambda_\Delta(D) \ge \beta$ uniform mesh refinement leads to a best
approximation error that satisfies
\begin{equation}
\label{eq_bestapprox}
\inf_{v_h \in V_h} \|u(\cdot,\omega)-v_h\|_{V} \le C_{{\rm FE}}(\omega) h^t
\,, \quad \text{for all} \ \ t < \beta\,,
\end{equation}
with $C_{{\rm FE}}(\omega) \sim \|u(\cdot,\omega) \|_{H^{1+t}(D)}$.
When $\lambda_\Delta(D) < \beta$, \eqref{eq_bestapprox} cannot be achieved
by uniform refinement. However, it can be recovered by a suitable local refinement.
For example, consider the 2D case where $D$ is smooth except for a single
reentrant corner with interior angle $\theta > \pi$ and where $W \subset D$
is a local neighbourhood of this  corner. Then $\lambda_\Delta(D) =
\pi/\theta$ and $u(\cdot,\omega) \in H^{1+t}(D \backslash W)$, for all $t
< \beta$, but $u(\cdot,\omega) \not\in H^{1+t}(W)$, for $\pi/\theta < t <
\beta$. However, by considering the best approximation error over $W$ and over
$D\backslash W$ separately, we see that it suffices to grade the
meshes such that the mesh size is
$\mathcal{O}(h^{\beta \theta/\pi})$ near the reentrant corner and
$\mathcal{O}(h)$
away from it. This is because
\begin{align*}
\inf_{v_h \in V_h} \|u(\cdot,\omega)-v_h\|_{V} &\le C_1
\|u(\cdot,\omega)\|_{H^{1+t}(D \backslash W)} h^t + C_2
\|u(\cdot,\omega)\|_{H^{1+\lambda_\Delta(D)}(D)} (h^{\beta\theta/\pi})^{\lambda_\Delta(D)}\nonumber
\end{align*}
for all $0<t<\beta$. Such a mesh grading can often be achieved
while retaining the desired complexity estimate $\text{dim}(V_h) \le C
h^{-2}$ (e.g., \cite{ScWa:79}).

Thus, using similar techniques to those in the proof of
\cite[Lem.~5.2]{TSGU13} it can be shown that \eqref{eq_bestapprox}  holds
with $C_{{\rm FE}}(\omega) \sim \|u(\cdot,\omega)\|_{H^{1+t}(D \backslash
W)} + \|u(\cdot,\omega)\|_{H^{1+\lambda_\Delta(D)}(D)}$, for all $t <
\beta$.\linebreak The case of multiple reentrant corners can be treated in
an identical fashion. Analogous but more complicated (anisotropic)refinement is needed in 3D, especially in the presence of
edge-singularities (e.g., \cite{Apel:1999}, \cite{BaNiZi:07}). In
practice, such local refinements can be constructed adaptively. The
important observation here is that the locally refined mesh needs to be
constructed only once for $\exp(\overline{Z}(\cdot))$ (or for one sample
of $a$), since the boundary singularities will be the same for all
samples.

We start our analysis by estimating the error  in approximating $a_\tau(\omega)$ by
$\widehat{a}_\tau(\omega)$.

\begin{lemma}
\label{lem:quad} Assume $a(\cdot,\omega) \in C^\beta(\overline{D})$ for
some $\beta \in (0,1]$. Furthermore,  for $h_0\sim h$ and for $\tau
\subseteq D$ with $\diam(\tau) \le h$, let $\cI_{h_0}$ and $Q_\tau$ be as
defined in \eqref{eq:multilin} and \eqref{eq:quad}, respectively. Then,
with $a_\tau(\omega)$ and $\widehat{a}_\tau(\omega)$ given by
\eqref{eq:exactint} and \eqref{eq:defatau},
\[
\vert a_\tau(\omega) - \widehat{a}_\tau(\omega) \vert
\,\le\,
|\tau| \, h^\beta \, \gamma^\beta \, |a(\cdot,\omega)|_{C^\beta(\overline{D})}
,
\]
with $\gamma = 1 + \sqrt{d}\,(h_0/h)$. If the quadrature points all lie on
the regular grid and we do not need interpolation, we may take $\gamma =
1$.
\end{lemma}

\begin{proof}
Using the fact that $a(\cdot,\omega)$ is continuous, the integral mean
value theorem asserts the existence of an $\bsx_\tau^* \in \tau$ such that
\begin{align*}
   a_\tau(\omega)
  \,=\,
  |\tau| \, a(\bsx_\tau^*,\omega)
  \,=\,
  \sum_{k=1}^r \sum_{i=1}^{2^d} \mu_{\tau,k} \, w_{i,\bsx_{\tau,k}} \, a(\bsx_\tau^*,\omega)
  ,
\end{align*}
where we used~\eqref{eq:multilin} and \eqref{eq:quad}. Then it follows
from \eqref{eq:approxhata} that
\begin{align*}
  \left| a_\tau(\omega) - \widehat{a}_\tau(\omega)\right|
  &\,=\,
  \Bigg| \sum_{k=1}^r \sum_{i=1}^{2^d} \mu_{\tau,k} \, w_{i,\bsx_{\tau,k}}
  \left( a(\bsx_\tau^*,\omega) - a(\bst_{i,\bsx_{\tau,k}},\omega)\right)\Bigg|\\
  &\,\le\,
  \sum_{k=1}^r \sum_{i=1}^{2^d} \mu_{\tau,k} \, w_{i,\bsx_{\tau,k}}
  \left| a(\bsx_\tau^*,\omega) - a(\bst_{i,\bsx_{\tau,k}},\omega)\right|
  \\
  &\,\le\,
  \sum_{k=1}^r \sum_{i=1}^{2^d} \mu_{\tau,k} \, w_{i,\bsx_{\tau,k}}\,
  \|\bsx_\tau^*-\bst_{i,\bsx_{\tau,k}}\|_2^\beta\,
  |a(\cdot,\omega)|_{C^\beta(\overline{D})}\\
  &\,\le\,
  \sum_{k=1}^r \sum_{i=1}^{2^d} \mu_{\tau,k} \, w_{i,\bsx_{\tau,k}}
  \left(\|\bsx_\tau^*- \bsx_{\tau,k}\|_2 + \|\bsx_{\tau,k} - \bst_{i,\bsx_{\tau,k}}\|_2\right)^\beta
  |a(\cdot,\omega)|_{C^\beta(\overline{D})}\\
  &\,\le\, |\tau| \, (h+\sqrt{d}h_0)^\beta \, |a(\cdot,\omega)|_{C^\beta(\overline{D})}\,.
\end{align*}
In the last step we used the fact that the distance between a point in a
cell of the regular grid and a vertex of that cell is at most
$\sqrt{d}\,h_0$. If the quadrature points all lie on the regular grid then
the second term can be omitted. This completes the proof.
\end{proof}

\begin{theorem}
\label{thm:H1error} Suppose that $Z(\cdot,\omega) \in
C^\beta(\overline{D})$ for some $\beta \in (0,1)$ and suppose that there
exists a family $\{\mathcal{T}_h\}_{h>0}$ of conforming, simplicial meshes
on $D$ such that \eqref{eq_bestapprox} holds with $\text{dim}(V_h) \le C
h^{-d}$. Let $\cI_{h_0}$ and $Q_\tau$ be defined in
\eqref{eq:multilin} and \eqref{eq:quad}, respectively. Then, we have
$\mathbb{P}$-almost surely
$$
\| u(\cdot,\omega) - u_h(\cdot,\omega)\|_{V} \ \le \
C_{{\cI Q}}(\omega) \, h^t\,, \quad \text{for all} \ \ t < \beta,
$$
with $C_{{\cI Q}}$ a positive random variable that satisfies
$\mathbb{E}[C_{{\cI Q}}^p] < \infty$, for all $p \in [1,\infty)$.

If $Z(\cdot,\omega) \in C^1(\overline{D})$ and \eqref{eq_bestapprox} also
holds for $t=1$, then $$\| u(\cdot,\omega) - u_h(\cdot,\omega)\|_{V} \
\le \ C_{{\cI Q}}(\omega) \, h.$$
\end{theorem}

\begin{proof}
The proof follows that of \cite[Prop.~3.13]{ChScTe:11}. First, using
Lemma~\ref{lem:quad} and the fact that $\nabla v_{h,\tau} := \nabla v_h
\vert_\tau$ is constant, for all piecewise linear finite element
functions $v_h \in V_h$, as well as applying the Cauchy--Schwarz
inequality in the last step we obtain the estimate
\begin{eqnarray*}
 \vert \mathscr{A}(\omega; w_h, v_h) - \mathscr{A}_h(\omega; w_h, v_h) \vert
  & =  & \sum_{\tau \in \cT_h}
\left\vert  a_\tau(\omega)  -  \widehat{a}_{\tau} (\omega)\right\vert \,
\big\vert (\nabla w_{h} \cdot \nabla
             v_{h})\vert_\tau \big\vert \\
  &   \le  &   h^\beta \, \gamma^\beta \, | a(\cdot, \omega)
         |_{C^\beta(\overline{D})}   \sum_{\tau \in \cT_h}  \vert
             \tau \vert  \, \big\vert \nabla w_{h,\tau} \cdot \nabla
             v_{h,\tau} \big\vert \\
&   \leq  &   h^\beta \, \gamma^\beta \, | a(\cdot, \omega)
          |_{C^\beta(\overline{D})} \,    \| v_h \|_{V}
            \| w_h \|_{V} \, .
\end{eqnarray*}
Now, using this bound in Strang's First Lemma
(cf.~\cite[Lem.~3.12]{ChScTe:11}), we can write
\begin{align} \label{eq_strang}
& \| u(\cdot,\omega)-u_h(\cdot,\omega)\|_{V} \nonumber\\
& \qquad \le \inf_{v_h \in V_h} \Bigg\{
\left( 1+\frac{a_{\mathrm{max}}(\omega)}{a_{\mathrm{min}}(\omega)}\right) \,
\|u(\cdot,\omega)-v_h\|_{V} + h^\beta \gamma^\beta\,
\frac{|a(\omega)|_{\mathcal{C}^{\beta}(\overline{D})}}{a_{\mathrm{min}}(\omega)} \,
\|v_h\|_{V} \Bigg\}\, .
\end{align}
Since
\[
\|v_h\|_{V} \le \|u(\cdot,\omega)-v_h\|_{V} +
\|u(\cdot,\omega)\|_{V} \le \|u(\cdot,\omega)-v_h\|_{V}
+ \frac{\|f\|_{L^2(D)}}{a_{\mathrm{min}}(\omega)} \,  ,
\]
we can combine \eqref{eq_strang} with \eqref{eq_bestapprox} to
establish the result.

The fact that the constant $C_{{\cI Q}}(\omega)$ in the above bounds
has bounded moments of any (finite) order is a consequence of our
assumptions that $Z(\bsx,\omega)$ is Gaussian and that
$Z(\cdot,\omega) \in \mathcal{C}^{\beta}(\overline{D})$. As stated
  above, it can be proved as in
\cite{ChScTe:11} via Fernique's Theorem.
\end{proof}

An $\mathcal{O}(h^{2t})$ bound on the $L^2$-norm of the error follows via
the well-known Aubin--Nitsche trick (cf.~\cite[Cor.~3.10]{ChScTe:11}). We
omit this and finish the section with an error bound for linear
functionals $\mathcal{G}$ of $u$, which we have already stated in \eqref{eq:functionals}.
\begin{corollary} \label{cor:functionals} Let $\mathcal{G}$ be a
bounded linear functional on $L^2(D)$. Then, under the assumptions of
Theorem~\textup{\ref{thm:H1error}}, there exists a constant $C > 0$ independent
of $h$ and $u$ such that
\[
\mathbb{E}\big[\vert \mathcal{G}(u) - \mathcal{G}(u_h) \vert\big]
\le C h^{2t}, \quad \text{for all} \ \ t < \beta\,.
\]
For $\beta = 1$, we get
$\mathbb{E}\big[\vert \mathcal{G}(u) - \mathcal{G}(u_h) \vert\big] \le
C h^{2}$.
\end{corollary}
\begin{proof} 
The proof follows, as in \cite[Lem.~3.3]{TSGU13}, from
H\"older's inequality using the fact that Theorem~\ref{thm:H1error}
applies verbatim also to the
FE error $\|z(\cdot,\omega)-z_h(\cdot,\omega)\|_V$ for the dual problem
$\mathscr{A}(\omega; v, z(\cdot,\omega)) = \mathcal{G}(v)$, for all $v
\in V$.
\end{proof}
Using the techniques in \cite[\S3]{TSGU13}, this corollary can be
extended in a straightforward way also to higher order moments
of the error or to functionals $\mathcal{G}$ of $u$ that are random,
nonlinear or bounded only on a subspace of $L^2(D)$.
In summary, we have provided in this section a recipe for extending all
results of \cite[\S3]{TSGU13} to general meshes, with the random field
being sampled on a regular grid and then interpolated onto the finite
element grid.


\section{QMC error analysis} \label{sec:qmc}

The QMC theory for integrals of the form \eqref{eq:unit_cube} is set in a
special weighted Sobolev space. Provided the integrand lies in this space,
we obtain an estimate for the root mean square error when a specially
chosen, randomly shifted lattice rule \eqref{eq:QMC} is used to
approximate \eqref{eq:unit_cube}. The cost for explicitly constructing a
good rule tailored to our analysis with $n$ points in $s$ dimensions grows
log-linearly in $n$ and quadratically in $s$ (cf.~Remark~\ref{rem:CBC}
below). However, applying the rule is essentially as cheap as obtaining
samples from a random number generator, see, e.g., \cite[\S7]{KuNu:15}.
Full details of the convergence theory are in other sources, e.g.,
\cite{GKNSSS:15, KuNu:15}, so we will be brief here.

Later in this section we use this theory to estimate the error when
the rule is applied to the particular $F$ given in \eqref{eq:int0}.
We assume first that the random field $Z$ is sampled by employing the quite
general factorisation \eqref{eq:Rfactor} of the covariance matrix $R$. Later,
in \S\ref{sec:expect} we will discuss the case when this is done by circulant
embedding.

\subsection{Abstract convergence result and proof strategy}

The relevant weighted Sobolev norm is defined as:
\begin{equation} \label{eq:norm2}
  \| F\|_{s,\bsgamma}^2
  \,:=\,
  \sum_{\setu\subseteq\{1:s\}} \frac{J_{\setu}(F)}{\gamma_\setu},
\end{equation}
where
\[
 J_{\setu} (F) \,:=\,
 \int_{\bbR^{|\setu|}}
  \bigg(
  \int_{\bbR^{s-|\setu|}}
  \frac{\partial^{|\setu|} F}{\partial
  \bsy_\setu}(\bsy_\setu;\bsy_{\{1:s\}\setminus\setu})\!\!
  \prod_{j\in\{1:s\}\setminus\setu} \phi(y_j)\,\rd\bsy_{\{1:s\}\setminus\setu}
  \bigg)^2
  \prod_{j\in\setu} \psi_j^2(y_j)
  \,\rd\bsy_\setu  \,.
\]
Here, $\{1 \! : \! s\}$ denotes the set  $\{1,2,\ldots,s\}$,
$\frac{\partial^{|\setu|}F}{\partial \bsy_\setu}$ denotes the mixed first
order derivative with respect to the ``active'' variables $y_j$ with
$j\in\setu$, $\bsy_{\{1:s\}\setminus\setu}$ denotes the ``inactive''
variables $y_j$ with $j\notin\setu$, and $\phi$ is the univariate normal
probability density (see \eqref{eq:unit_cube}). The remaining ingredients
in \eqref{eq:norm2} are the {\em weight parameters} $\gamma_\setu$
and the \emph{weight functions} $\psi_j$, which are used, respectively,
to moderate the relative importance of the derivatives of $F$ with
respect to $\bsy_\setu$ and to control the behaviour of these
derivatives asymptotically as $\Vert
\bsy \Vert_\infty \rightarrow \infty$. As in \cite{GKNSSS:15}, we shall
restrict ourselves to the choice
\begin{equation} \label{eq:psi}
  \psi_j^2(y_j) \,=\, \exp (-2\,\alpha_j\, |y_j|)\;,
  \quad\mbox{for some }\quad
   \alpha_j > 0 \ .
\end{equation}

The following result is then essentially \cite[Theorem~15]{GKNSSS:15} (see
also \cite{KuNu:15}).

\begin{theorem} \label{thm:QMC2}
Suppose $\Vert F \Vert_{s, \bsgamma} < \infty$ and $n$ is a power of a
prime. Then, a generating vector $\bsz
\in \mathbb{N}^s$ for a randomly shifted lattice rule \eqref{eq:QMC}
can be constructed so that the root mean square error in
applying \eqref{eq:QMC} to \eqref{eq:unit_cube} satisfies
\begin{align}
\label{eq:basicQMC}
 \sqrt{\bbE_\bsDelta\left[ |I_s(F) - Q_{s,n}(\bsDelta,F)|^2 \right]}
 \,\le\,
 \left(\frac{2}{n}\right)^{1/(2\kappa)}  \,
 \wC_s(\bsgamma, \bsalpha, \kappa)
   \
   \|F\|_{s,\bsgamma}\,,
\end{align}
for all $\kappa\in (1/2,1]$, where
\begin{equation} \label{eq:tilde-C}
\wC_s(\bsgamma, \bsalpha, \kappa) \ = \
\Bigg(
  \sum_{\emptyset\ne \setu\subseteq\{1:s\}}
  \gamma_\setu^{\kappa}\, \prod_{j\in\setu}
  \varrho(\alpha_j,\kappa)\Bigg)^{1/(2\kappa)}\, ,
\end{equation}
\begin{equation} \label{eq:rho-j}
  \varrho(\alpha,\kappa)
  \,=\, 2\left(\frac{\sqrt{2\pi}\,\exp(\alpha^2/\eta(\kappa))}{\pi^{2-2\eta(\kappa)}(1-\eta(\kappa))\eta(\kappa)}\right)^{\kappa}\,
  \zeta\big(\kappa + \tfrac{1}{2}\big),
  \qquad
  \eta(\kappa) \,:=\, \frac{2\kappa-1}{4\kappa}\,,
\end{equation}
and $\zeta$ is the Riemann zeta function.
\end{theorem}

A method for constructing the vector $\bsz$ one component at a time is
described in \cite{NK14} for weights $\bsgamma_\setu$ of a special form,
see Remark~\ref{rem:CBC} below.\bigskip

In the remainder of this section we shall apply this theory to the
function $F$ given in \eqref{eq:int0}. The main result is
Theorem~\ref{thm:QMC}. It is obtained in the steps summarised as
follows.
\begin{enumerate}
\item 
By differentiating the parametrised discrete weak form \eqref{fepar},
we estimate the norms $\|\left(\partial^{\vert \setu\vert }
u_h/\partial\bsy_\setu \right) (\cdot,\bsy) \|_V $  for any $\setu \in
\{1:s\}$.  The estimate (which uses  an induction argument over the
set of all partial derivatives of $u_h(\cdot, \bsy)$) is given in
Theorem \ref{thm:reg} and involves the quantities
\begin{equation} \label{eq:bj}
  b_j \,:=\,  \Vert \bsB_j\Vert_\infty,
  \qquad j=1,\ldots,s,
\end{equation}
where $\bsB_j$ is the $j$th column of the $M\times s$ matrix $B$
introduced in \eqref{eq:Rfactor}. We let $\bsb \in \bR^s$ be the
vector $(b_1,\ldots,b_s)^\top$.

\item 
Using the result from Step~1 and the linearity of $\calG$, we estimate
$\|F\|_{s, \bsgamma}$ in Theorem~\ref{thm:norm}. The shape parameters
$\alpha_j$ from \eqref{eq:psi} are constrained to be in the range
$\alpha_j>b_j$. The precise values of $\alpha_j$ are arbitrary at this
point, and so are the values of the weight parameters
$\gamma_\setu$.

\item 
We substitute the result from Step~2 into the right hand side of
\eqref{eq:basicQMC} to obtain an error bound for this particular
$F$, and we choose the weight parameters $\gamma_\setu$ and then
the shape parameters $\alpha_j$ to minimise this bound. The end
result, Theorem~\ref{thm:QMC},  is a convergence estimate with order
$\calO(n^{-1/(2\kappa)})$, valid for $\kappa\in (1/2,1]$, and with
the implied constant depending on the sum $\sum_{j=1}^s
b_j^{2\kappa/(1+\kappa)}$.
\end{enumerate}

The theory is essentially independent of the choice of factorisation
\eqref{eq:Rfactor}. However, we will work out the detailed theory, in
\S\ref{sec:expect} and \S\ref{sec:Numerical}, only for the
circulant embedding approach.

\subsection{Regularity of $F$}

In this subsection it is helpful to introduce more general partial
derivatives than those mixed first order derivatives which appear in
\eqref{eq:norm2}. Thus for any multiindex $\bsnu \in \bN^s$ with order
$\vert \bsnu \vert = \sum \nu_j$, we let $\partial^{\bsnu}$ denote the
corresponding mixed derivative. For any multiindex $\bsnu \in \bN^s$ and
any vector $\bsc \in \bR^s$ we also write $\bsc^\bsnu =
\textstyle\prod_{j=1}^s c_j^{\nu_j}$.

\begin{theorem} \label{thm:reg}
For any $\bsy\in \bbR^s$, any $f\in V'$, and for any multiindex $\bsnu \in
\bN^s$, the solution $u_h(\cdot,\bsy)$ of \eqref{fepar} satisfies
\begin{equation*}
  \|\partial^{\bsnu} u_h(\cdot,\bsy) \|_V
  \,\le\,  | \bsnu |!
  \bigg( \frac{\bsb}{\log 2} \bigg)^\bsnu
  \frac{1}{{a}_{\min,M}(\bsy)}  \|f\|_{V'} \;,
\end{equation*}
where $\bsb = (b_1,\ldots,b_s)^\top$ is defined in \eqref{eq:bj} and
$a_{\min,M} (\bsy) = \min_{1\le i\le M} a(\bsx_i, \bsy)$.
\end{theorem}

\begin{proof}
The proof is similar to that of \cite[Theorem~14]{GKNSSS:15}, but there
are some differences due to the fact that we are working with the FE
discretisation \eqref{fepar} with quadrature and interpolation, and
because of the finiteness of the `expansion' \eqref{eq:Gauss_exp1}. (In
\cite{GKNSSS:15} an infinite KL expansion was used in the context of the
continuous problem.)

To simplify the proof we introduce the $\bsy$-dependent discrete norm
$\interleave \cdot\interleave$ on $V_h$:
\begin{equation}\label{eq:discnorm}
  \interleave v_h \interleave_{\bsy}^2
  \ := \  \sum_{\tau \in \cT_h} \ha_\tau(\bsy) \left. |\nabla v_h|^2 \right\vert_\tau\ ,
  \quad v_h \in V_h\ ,
\end{equation}
with $\hat{a}_\tau(\bsy)=\hat{a}_\tau(\omega)$ given by
\eqref{eq:defatau}. Then we have $\interleave v_h \interleave_{\bsy}^2 =
\scA_h(\bsy;v_h,v_h)$, see \eqref{eq:defAh}. Since we consider
piecewise linear finite elements, for $v_h\in V_h$ we have that $\nabla
v_h$ is piecewise constant on each element $\tau$.

We first prove by induction on $|\bsnu|$ that the solution $u_h(\cdot,
\bsy)$ of \eqref{fepar} satisfies
\begin{align} \label{eq:hypo}
  \interleave \partial^{\bsnu}u_h(\cdot, \bsy)  \interleave_{\bsy} \
  &\,\le\, \Lambda_{|\bsnu|}\, \bsb^{\bsnu}\,
  \interleave u_h(\cdot, \bsy)  \interleave_{\bsy}
  \ ,
  \quad \bsy \in \bR^s \ ,
\end{align}
where the sequence $(\Lambda_n)_{n\ge 0}$ is defined recursively by
$$\Lambda_0 \,:=\, 1 \quad \text{and} \quad  \Lambda_n \,:=\, \sum_{i=0}^{n-1} \binom{n}{i}
\Lambda_i, \quad \text{ for all} \quad  n\ge 1 \ .$$

Clearly \eqref{eq:hypo} holds for $|\bsnu| = 0$. For $\bsnu\ne\bszero$, we
differentiate \eqref{fepar}, using  the multivariate Leibniz rule to
obtain (since the right-hand side is independent of $\bsy$),
\begin{equation}\label{eq:leib}
  \sum_{\bsm\le\bsnu} \binom{\bsnu}{\bsm} \sum_{\tau \in \cT_h}
  \big(\partial^{\bsnu- \bsm} \ha_\tau(\bsy) \big)
  \big(\nabla \partial^{\bsm} u_h(\cdot,\bsy) \cdot  \nabla v_h\big)\big|_\tau \ = \ 0 ,
  \quad \text{for all} \quad v_h \in V_h .
\end{equation}
Now inserting $v_h = \partial^{\bsnu} u_h(\cdot,\bsy)$ into
\eqref{eq:leib}, keeping the term with $\bsm = \bsnu$ in the outer sum  on
the left-hand side and moving the remaining terms to the right-hand side,
we have
\begin{align*}
 &\interleave  \partial^{\bsnu} u_h(\cdot,\bsy)\interleave_{\bsy}^2
 \,=\, - \sum_{\satop{\bsm\le\bsnu}{\bsm\ne\bsnu}} \binom{\bsnu}{\bsm}
  \sum_{\tau \in \cT_h} \big(\partial^{\bsnu- \bsm}\ha_\tau(\bsy) \big)
 \big(\nabla \partial^{\bsm} u_h(\cdot,\bsy) \cdot \nabla \partial^{\bsnu} u_h(\cdot,\bsy) \big)\big|_\tau \nonumber\\
 &\,\le\,
\sum_{\satop{\bsm\le\bsnu}{\bsm\ne\bsnu}} \binom{\bsnu}{\bsm}
 \left(\max_{\tau\in\cT_h}
 \left|
 \frac{\partial^{\bsnu- \bsm} \ha_\tau(\bsy)}{\ha_\tau(\bsy)}\right|
 \right)  \,
  \sum_{\tau \in \cT_h} \ha_\tau(\bsy)\, \big|
  \big(\nabla\partial^{\bsm} u_h(\cdot,\bsy) \cdot \nabla\partial^{\bsnu} u_h(\cdot,\bsy)\big) \big|_\tau  \big| .
\end{align*}
Then, after an application of the Cauchy--Schwarz inequality and a
cancellation, we obtain
\begin{align} \label{eq:CS}
  \interleave \partial^{\bsnu} u_h(\cdot,\bsy)\interleave_{\bsy}
  &\,\le\,
 \sum_{\satop{\bsm\le\bsnu}{\bsm\ne\bsnu}} \binom{\bsnu}{\bsm}
 \left(\max_{\tau\in\cT_h}
 \left|
 \frac{\partial^{\bsnu- \bsm} \ha_\tau(\bsy)}{\ha_\tau(\bsy)}\right|
 \right) \,
 \interleave \partial^{\bsm} u_h(\cdot,\bsy)\interleave_{\bsy}  \ .
\end{align}

To estimate \eqref{eq:CS}, we have from Proposition~\ref{prop1} that
$\ha_\tau(\bsy) = \sum_{i=1}^M p_{\tau,i} \,a(\bsx_i,\bsy)$ with all
$p_{\tau,i}\ge 0$, and we recall from \eqref{eq:field} and
\eqref{eq:Gauss_exp1} that $a(\bsx_i,\bsy) = \exp(\sum_{j=1}^s B_{i,j}y_j
+ \overline{Z}_i)$. Then, noting that $a(\bsx_i,\bsy)\ge 0$ it is easy to
see that, for any multiindex $\bsnu$,
$$
  |\partial^{\bsnu} a(\bsx_i,\bsy)|
  \ = \ a(\bsx_i,\bsy)  \prod_{j=1}^{s} |B_{i,j}^{\nu_j}| \,
 \ \leq \ a(\bsx_i,\bsy)\,\bsb^\bsnu
  \ ,
 $$
which leads to $\vert \partial^{\bsnu} \ha_\tau(\bsy) \vert \leq
\ha_\tau(\bsy) \bsb^\bsnu$, and hence
\[
  \max_{\tau\in\cT_h}
 \left|
 \frac{\partial^{\bsnu- \bsm} \ha_\tau(\bsy)}{\ha_\tau(\bsy)}\right|
 \,\le\, \bsb^{\bsnu-\bsm}\,.
\]
Inserting this into \eqref{eq:CS} we obtain
\begin{align}
 \interleave \partial^{\bsnu} u_h(\cdot,\bsy)\interleave_{\bsy}
  &\,\le\,
 \sum_{\satop{\bsm\le\bsnu}{\bsm\ne\bsnu}} \binom{\bsnu}{\bsm}  \, \bsb^{\bsnu-\bsm}
 \interleave \partial^{\bsm} u_h(\cdot,\bsy)\interleave_{\bsy}  \ .
 \label{eq:CS1}
\end{align}
Using \eqref{eq:CS1}, the estimate \eqref{eq:hypo} then follows by
induction in exactly the same way as in \cite[Theorem~14]{GKNSSS:15}.

Now, using the definition of the discrete norm \eqref{eq:discnorm} and the
fact that $u_h(\cdot,\bsy)$ is the solution of \eqref{fepar}, we have
\begin{equation} \label{eq:connect1}
 \left(\min_{\tau \in \cT_h}\frac{\ha_\tau(\bsy)}{\vert \tau\vert} \right)  \Vert u_h(\cdot, \bsy)\Vert_V^2
 \,\le\, \interleave u_h(\cdot, \bsy)\interleave^2_{\bsy} = \langle f, u_h(\cdot, \bsy)\rangle
 \,\le\, \Vert f \Vert_{V'} \Vert u_h(\cdot, \bsy)\Vert_V \ .
\end{equation}
Hence, using Proposition~\ref{prop1}, we conclude that
\begin{equation} \label{eq:connect2}
 \Vert u_h(\cdot, \bsy)\Vert_V \,\le\, \frac{\Vert f \Vert_{V'}}{{a}_{\min,M}(\bsy)}
 \quad\mbox{and}\quad
 \interleave u_h(\cdot,\bsy)\interleave_{\bsy}
 \,\le\, \frac{\Vert f \Vert_{V'}}{\sqrt{{a}_{\min,M}(\bsy)}}
 \ .
\end{equation}
Using the same argument as for the lower bound in \eqref{eq:connect1},
together with Proposition~\ref{prop1} again, we obtain from
\eqref{eq:hypo} and \eqref{eq:connect2} that
\[
  \sqrt{a_{\min,M}(\bsy)}\, \Vert \partial^{\bsnu} u_h(\cdot, \bsy)\Vert_V
  \,\le\,
  \interleave \partial^{\bsnu} u_h(\cdot,\bsy)\interleave_{\bsy}
  \,\le\, \Lambda_{|\bsnu|}\,\bsb^\bsnu\, \frac{\Vert f \Vert_{V'}}{\sqrt{{a}_{\min,M}(\bsy)}}.
\]
This together with the estimate $\Lambda_n \ \leq n!/(\log 2)^n$ (proved
in \cite[Theorem 14]{GKNSSS:15}) completes the proof.
\end{proof}

\vspace{0.2cm}

We can now use this theorem to show that $F$ lies in the weighted Sobolev
space characterised by the norm \eqref{eq:norm2}. We make use of the
$s$-dependent quantities
$$
\Vert \bsb \Vert_{p,s} = \Bigg(\sum_{j=1}^s  \vert b_j\vert^p \Bigg)^{1/p}, \quad p > 0\ ,
\quad  \text{and} \quad
\Vert \bsb \Vert_{\infty,s} = \max_{j \in  \{1:s\}} \vert b_j\vert \ .
$$

\begin{theorem} \label{thm:norm}
Suppose that $\Vert \bsb \Vert_{1,s}$ is uniformly bounded with respect to
$s$. Suppose also that  $\alpha_j > b_j$ for all $j$. Then, for any $f\in
V'$ and for any linear functional $\calG\in V'$, the integrand $F(\bsy) =
\calG(u_h(\cdot,\bsy))$ in \eqref{eq:int0} satisfies
\begin{equation*}
 {\Vert F \Vert_{s, \bsgamma}}
 \,\le \, C
 \Bigg(\sum_{\setu\subseteq \{1:s\}}
\frac{1}{\gamma_\setu}
\left(\frac{\vert \setu \vert!}{(\log 2)^{\vert \setu \vert}}\right)^2
\prod_{j \in \setu}
\frac{\widetilde{b}_j^2}{\alpha_j - b_j}\Bigg)^{1/2}\, ,
\end{equation*}
where $C$ is independent of $s$ and
\[
  \widetilde{b}_j = \frac{b_j}{2 \exp(b_j^2/2) \Phi(b_j)}\,,
\]
with $\Phi$ denoting the univariate standard cumulative normal
distribution function.
\end{theorem}
\begin{proof}
Using the linearity of $\calG$, together with Theorem~\ref{thm:reg}, but
replacing the multiindex $\bsnu$ with any set $\setu \subseteq \{1:s\}$
(i.e., restricting to the case where all $\nu_j\le 1$), we obtain the
following estimate for the first order partial derivatives of $F$
appearing in the norm \eqref{eq:norm2},
\begin{align} \label{eq:regF}
  \left\vert \frac{\partial^{\vert \setu\vert} F}{\partial \bsy_\setu}(\bsy) \right\vert
  &\,\le\,
  \frac{| \setu |!}{(\log 2)^{\vert \setu \vert}}
  \bigg(\prod_{j\in\setu} b_j\bigg)
  \frac{1}{a_{\min,M}(\bsy)}\, \|f\|_{V'}\, \Vert \cG \Vert_{V'} \nonumber\\
  &\,\le\,
  \frac{| \setu |!}{(\log 2)^{\vert \setu \vert}}
  \bigg(\exp(\bsb^\top \vert \bsy \vert) \prod_{j\in\setu} b_j\bigg)
  \bigg(\exp(\Vert \bZbar \Vert_\infty)\, \|f\|_{V'}\, \Vert \cG \Vert_{V'} \bigg)  \ ,
\end{align}
where we used the estimate $a_{\min,M}(\bsy) = \min_{1\le i\le M}
a(\bsx_i, \bsy) \ge \exp(-\Vert \bZbar \Vert_\infty)\exp(- \bsb^{\tr}
\vert \bsy \vert)$.

Examining the right-hand side of \eqref{eq:regF}, we see that the only
factor which depends on $\bsy$ is $\exp(\bsb^\top|\bsy|)$. An elementary
calculation (see \cite[Theorem 16]{GKNSSS:15}) shows that
\begin{align*}
  &\int_{\bbR^{|\setu|}}
  \bigg(
  \int_{\bbR^{s-|\setu|}}
  \bigg(\exp(\bsb^\top|\bsy|)\prod_{j\in\setu} b_j\bigg)
  \prod_{j\in\{1:s\}\setminus\setu} \phi(y_j)\,\rd\bsy_{\{1:s\}\setminus\setu}
  \bigg)^2
  \prod_{j\in\setu} \psi_j^2(y_j)
  \,\rd\bsy_\setu  \\
  &\,=\,
  \Bigg( \prod_{j\in \{1:s\}\backslash \setu } \left(2 \Phi(b_j) \exp(b_j^2/2)\right)^2  \Bigg)
  \Bigg( \prod_{j \in \setu} \frac{{b}_j^2}{\alpha_j-b_j}\Bigg) \\
  &\,=\,
  \Bigg( \prod_{j\in \{1:s\} } \left(2 \Phi(b_j) \exp(b_j^2/2)\right)^2 \Bigg)
  \Bigg( \prod_{j \in \setu} \frac{{\widetilde{b}}_j^2}{\alpha_j-b_j}\Bigg) \ .
\end{align*}
Since $2\Phi(b_j)\le 1 + 2b_j/\sqrt{2\pi} < \exp(b_j)$ for all $j$, we
have
\[
  \prod_{j\in \{1:s\} } \left(2 \Phi(b_j) \exp(b_j^2/2)\right)^2
  \,\le\, \exp(2\|\bsb\|_{1,s} + \|\bsb\|_{2,s}^2)\,.
\]
Thus, it follows from \eqref{eq:regF} and the definition of the norm
\eqref{eq:norm2} that
\begin{align*}
 \frac{\Vert F \Vert_{s, \bsgamma}}{\exp(\Vert \bZbar \Vert_\infty)\, \Vert f\Vert_{V'}  \Vert \cG\Vert_{V'}}
 &\,\le\, \exp(2\|\bsb\|_{1,s} + \|\bsb\|_{2,s}^2)
 \Bigg( \sum_{\setu\subseteq \{1:s\}}
\frac{1}{\gamma_\setu}
\left(\frac{\vert \setu \vert!}{(\log 2)^{\vert \setu \vert}}\right)^2
\prod_{j \in \setu}
\frac{\widetilde{b}_j^2}{\alpha_j - b_j}\Bigg)^{1/2}.
\end{align*}
The final result, with the constant factor $C$ being independent of
$s$, is then a consequence of the assumption on $\bsb$.
\end{proof}

\subsection{Error estimate}
\label{subsec:errorest}

In order to obtain a dimension-independent estimate for the QMC method we
need a stronger assumption on $\bsb$ than that used in
Theorem~\ref{thm:norm}. The following theorem shows that under that
stronger assumption there is a choice of $\bsgamma$ and $\bsalpha$ which
ensures that the QMC error is bounded independently of $s$. The
appropriate choice of $\bsgamma$ is of ``POD'' type,  which allows a good
generating vector $\bsz$ for the QMC rule to be efficiently computed by
the ``component-by-component''  procedure, see Remark~\ref{rem:CBC} below.

\begin{theorem}\label{thm:QMC}
Under the assumptions of Theorem~\textup{\ref{thm:norm}}, let $\kappa \in
(1/2,1)$, set $p = 2 \kappa/(1+\kappa)$ and assume in
  addition that $\Vert
\bsb\Vert_{p,s}$  is uniformly bounded with respect to $s$. Then there
exists a positive constant $C(\kappa)$ depending on $\kappa$ (as well as
on $\bZbar$, $f$, $\calG$) such that
\begin{equation}\label{eq:RMSbd}
  \sqrt{\bbE_\bsDelta [|\,I_s(F) - Q_{s,n}(\bsDelta,F)|^2]}
  \,\leq \,  C(\kappa) \, n^{-1/(2\kappa)}\,.
\end{equation}
\end{theorem}

\begin{proof}
Some parts of the proof are similar to that of \cite[Theorem
20]{GKNSSS:15}, and for these parts we will be brief. We remind the
readers that each vector $\bsb = (b_1,\ldots,b_s)^\top$ depends
fundamentally on $s$; changing the value of $s$ leads to completely
different components $b_j$. This is different from the situation in
\cite{GKNSSS:15} where there is just one infinite sequence $\bsb$ that is
truncated to $s$ terms.

First, combining Theorem \ref{thm:QMC2} and Theorem \ref{thm:norm} we see
that \eqref{eq:RMSbd} holds with  $C(\kappa)$ proportional to
 \begin{align*}
  &{C}_s(\bsgamma, \bsalpha,\kappa)
  \,:=\, \wC_s(\bsgamma,\bsalpha,\kappa)
  \Bigg(\sum_{\setu\subseteq\{1:s\}}
  \frac{1}{\gamma_\setu } \left(\frac{|\setu|!}{\,{(\log 2)^{|\setu|}}}\right)^2\,
  \prod_{j\in\setu}
  \frac{\widetilde{b}_j^2}{\alpha_j-b_j}
  \Bigg)^{1/2},\,
\end{align*}
with $\wC_s(\bsgamma,\bsalpha,\kappa)$ defined in \eqref{eq:tilde-C}.
Now, we choose the weight parameters $\bsgamma$ to minimise $C_s(\bsgamma,
\bsalpha, \kappa)$. This minimisation problem was solved in \cite[Lemma
18]{GKNSSS:15}, yielding the solution:
\begin{align} \label{eq:gamma}
  \gamma_\setu
  \, = \, \gamma_\setu^*  \,:=\,
  \bigg( \bigg(\frac{|\setu|!}{(\log 2)^{|\setu|}}\bigg)^2
  \prod_{j\in\setu}
  \frac{\widetilde{b}_j^2}{(\alpha_j-b_j)\,\varrho(\alpha_j, \kappa)}
  \bigg)^{1/(1+\kappa)},
\end{align}
which is of ``product and order dependent'' (POD) form. With this choice,
one can show that
$$
C_s(\bsgamma^*, \bsalpha, \kappa) \ = \ S_s(\bsalpha, \kappa)^{(\kappa+1)/(2 \kappa)}\ ,
$$
where
\begin{equation} \label{eq:Slambda}
 {S}_s(\bsalpha,\kappa)
 \,=\,
 \sum_{\setu\subseteq\{1:s\}}
  \bigg( \bigg(\frac{|\setu|!}{(\log 2)^{|\setu|}}\bigg)^2 \,
  \prod_{j\in\setu} \frac{\varrho^{1/\kappa}(\alpha_j,\kappa)\,\widetilde{b}_j^2}{\alpha_j-b_j}\bigg)^{\kappa/(1+\kappa)}
  \,.
\end{equation}

It remains to estimate $S_s(\bsalpha, \kappa)$. Apart from the
constraint $\alpha_j > b_j$, the shape parameters $\alpha_j$ are
still free at this stage and so we choose them to minimise the right-hand
side of \eqref{eq:Slambda}. This minimisation problem is also solved in
\cite[Corollary 21]{GKNSSS:15}); the solution is
\begin{equation} \label{eq:alpha}
  \alpha_j
  \,:=\,
  \frac{1}{2}\bigg(b_j + \sqrt{b_j^2+1-\frac{1}{2\kappa}}\,\bigg).
\end{equation}

Now, to estimate $S_s(\bsalpha,\kappa)$, let $b_{\max}$ be an upper bound
on $\Vert \bsb\Vert_{\infty,s}$ for all $s$ (guaranteed by assumption).
Then $b_j \,\le\, b_{\max}$ for all $j=1,\ldots,s$ and all $s$. For a
given value of $\kappa\in (1/2,1]$, let $\alpha_{\max}$ denote the value
of \eqref{eq:alpha} with $b_j$ replaced by $b_{\max}$. Then we have
$\alpha_j \le \alpha_{\max}$ for all $j=1,\ldots,s$ and all $s$, and
\begin{align*}
  {\alpha_j - b_j}
  \,&=\, \frac{1}{2} \frac{1-1/(2\kappa)}
{\sqrt{b_j^2+1-1/(2\kappa)} + b_j}
  \, \geq \, \frac{1}{2} \frac{1-1/(2\kappa)}
{\sqrt{b_{\max}^2+1-1/(2\kappa)} + b_{\max}}
\, =  \,  {\alpha_{\max} - b_{\max}}.
\end{align*}
Note also that for $\varrho$ defined in \eqref{eq:rho-j} we have
$\varrho(\alpha_j,\kappa) \le \varrho(\alpha_{\max},\kappa)$ for all $j$
and all $s$. Moreover, since $2\exp(b_j^2/2)\Phi(b_j)\ge 1$, it follows
that $\widetilde{b}_j \leq b_j$, and so from \eqref{eq:Slambda} we can
conclude that, with $p := 2\kappa/(1+\kappa)$,
\begin{equation*}
 S_s(\bsalpha,\kappa) \,\le\, \sum_{\setu\subseteq\{1:s\}}
  (|\setu|!)^{p}
  \prod_{j\in\setu} (\tau_{\kappa}\, b_j^2)^{p/2}  \ = \
  \sum_{\ell = 0}^s (\ell!)^{p} \!\!\sum_{\setu\subseteq\{1:s\},\,|\setu|=\ell}\;
  \prod_{j \in \setu} (\tau_{\kappa} b_j^2)^{p/2}
  \;,
\end{equation*}
where
\[
  \tau_{\kappa} \,:=\, \frac{\varrho^{1/\kappa}(\alpha_{\max},\kappa)}{(\log 2)^2(\alpha_{\max} - b_{\max})} .
\]
Now using the inequality
$$
  \ell! \, \sum_{\satop{\setu\subseteq\{1:s\}}{|\setu|=\ell}} \prod_{j\in \setu} a_j
  \leq \bigg(\sum_{j=1}^s a_j\bigg)^\ell ,
$$
(which holds since for $|\setu| = \ell$ each term $\prod_{j\in\setu} a_j$
from the left-hand side appears in the expansion of the right-hand side
exactly $\ell!$ times, but the right-hand side includes other terms) and
with $K$ denoting the assumed uniform bound on $\Vert \bsb
\Vert_{p,s}^p$, we obtain
\begin{equation*}
 S_s(\bsalpha,\kappa) \,\le\,
\sum_{\ell = 0}^s (\ell!)^{p -1}
\tau_{\kappa}^{p\ell/2}  \bigg(\sum_{j=1}^s  b_j^{p} \bigg)^\ell
\le\,
\sum_{\ell = 0}^\infty (\ell!)^{p -1}
\tau_{\kappa}^{p\ell/2}  K^\ell
\,<\,\infty\,.
\end{equation*}
The finiteness of the right-hand side follows by the ratio test, on noting
that $p<1$.
\end{proof}

\begin{remark}\label{rem:CBC}
A generating vector $\bsz\in\bbN^s$ for a randomly shifted lattice rule
with $n$ points in $s$ dimensions that achieves the desired error bound
can be constructed using a component-by-component (CBC) algorithm, which
goes as follows: (1) Set $z_1 = 1$. (2) For each $k = 2,3,\ldots, s$,
choose $z_k$ from the set $\{1\le z\le n-1: \gcd(z,n)=1\}$ to minimise
\[
  E^2_{s,n,k}(z_1,\ldots,z_k) \,:=\,
 \sum_{\emptyset\ne\setu\subseteq\{1:k\}}
 \frac{\gamma_\setu}{n}\sum_{i=1}^n  \prod_{j\in\setu}\theta_j\left({\rm frac}\left( \frac{iz_j}{n} \right)\right),
\]
where the function $\theta_j(x)$ is symmetric around $1/2$ for
$x\in[0,1]$ and can be computed for $x\in [0,1/2]$ by
\[
  \theta_j(x)
  \,:=\,
  \frac{x - \frac12 +
    \exp(2 \alpha_j^2)
    \left[
      \Phi(2 \alpha_j)
      -
      \Phi\big(2\alpha_j+\Phi^{-1}(x)\big)
    \right]}{\alpha_j}
  -2 \int_{-\infty}^0 \frac{\Phi(t)^2}{\psi_j(t)^2} \rd{t}.
\]
The integral in the above formula for $\theta_j$ only needs
to be calculated once, while the general formulation \cite[Equation~(50)]{NK14}
also has an integral for the first part, which we evaluate explicitly
here for our particular choice of $\psi_j$.
Note that $\gamma_\setu$ and $\alpha_j$ fundamentally depend
on $s$ through $b_j$. When (and only when) the algorithm reaches $k=s$, the
expression $E^2_{s,n,s}(z_1,\ldots,z_s)$ is the so-called \emph{squared
shift-averaged worst case error}. See \cite{NK14} for the analysis of an
efficient implementation of the algorithm for POD weights
\eqref{eq:gamma}, so that the cost is $\calO(sn\log n + s^2n)$ operations
using FFT. We refer to the accompanying software of \cite{KuNu:15} for
an implementation.
\end{remark}


\subsection{QMC convergence in the case of circulant embedding}
\label{sec:expect}

The circulant embedding technique is a method of computing efficiently the
factorisation \eqref{eq:Rfactor}, thus yielding a method of sampling the
random vector $\bsZ$ via \eqref{eq:Gauss_exp1}. We describe the process
briefly here before verifying the assumptions of Theorem
\ref{thm:QMC}. This section is a summary of our results in \cite{paper1}.

The $M = (m_0+1)^d$ points $\{\bsx_i: i = 1, \ldots, M\}$ are assumed to
be uniformly spaced with spacing $h_0 :=1/m_0$ on a $d$-dimensional grid
over the unit cube $[0,1]^d$ enclosing the domain $D$. Using a vector
notation, we may relabel the points $\bsx_1,\ldots,\bsx_M$ to be indexed
by $\bk$ as
\begin{equation*}
 \bsx_{\bk} \,:=\, h_0 \bk  \qquad\text{for}\quad \bk = (k_1, \ldots, k_d)  \in \{0 , \ldots, m_0\}^d .
\end{equation*}
Then it is easy to see that (with analogous vector notation for the rows
and columns) the $M \times M$ covariance matrix $R$ defined in
\eqref{eq:Rmatrix} can be written as
\begin{equation} \label{eq:defR}
  R_{\bk, \bk'}
  \,=\, \rho\big(h_0 (\bk-\bk') \big),
  \qquad \bk, \bk'  \in \{ 0, \ldots, m_0\}^d  .
\end{equation}
If the vectors $\bk$ are enumerated in lexicographical ordering, then we
obtain a nested block Toeplitz matrix where the number of nested levels is
the physical dimension $d$.

We extend $R$ to a nested block circulant matrix $\Rext$. To do this,
it is convenient to extend to the infinite grid:
\begin{equation*}
 \bsx_{\bk} \,:=\, h_0 \bk  \qquad\text{for}\quad \bk \in \bbZ^d .
\end{equation*}
Then, to define  $\Rext$, we consider an enlarged cube $[0,\len]^d$ of
edge length $\len:=mh_0\ge 1$ with integer $m\ge m_0$. We assume that
$m_0$ (and hence $h_0$) is fixed and we enlarge $m$ (or equivalently
$\ell$) as appropriate. We introduce a $2 \len$-periodic map on $\bbR$ by
specifying its action on $[0,2 \len]$:
\begin{equation*}
  \varphi(x)
  \,:=\,
  \begin{cases}
    x          & \text{if}\quad 0\,\le\, x  \,\le\, \len ,\\
    2\len - x  & \text{if}\quad \len \,\le\, x  \,<\, 2\len .
  \end{cases}
\end{equation*}%
Now we apply this map elementwise and define an extended version
$\rhoext$ of $\rho$ as follows:
\begin{equation*}
 \rhoext(\bsx)
 \,:=\,
 \rho(\varphi(x_1),\ldots,\varphi(x_d))
 ,\qquad \bsx\in \bbR^d .
\end{equation*}
Note that $\rhoext$ is $2\len$-periodic in each coordinate direction and
$\rhoext(\bsx) = \rho(\bsx)$ when $ \bsx \in [0,\len]^d$. Then $\Rext$ is
defined to be the $s\times s$ symmetric nested block circulant matrix with
$s = (2m)^d$, defined, analogously to \eqref{eq:defR}, by
\begin{equation}\label{eq:Rext_def}
 \Rext_{\bk, \bk'}\,=\, \rho^{\rm{ext}}\big( h_0 (\bk - \bk') \big),
 \qquad \bk, \bk' \in \{0, \ldots,  2m-1\}^d  .
\end{equation}
It follows that $R$ is the submatrix of $\Rext$ in which the indices are
constrained to lie in the range $\bk, \bk' \in \{ 0,\ldots, m_0\}^d$.
Since $\Rext$ is nested block circulant, it is diagonalisable by FFT.
The following theorem is taken from \cite{GrKuNuScSl:11}:

\begin{theorem} \label{thm:decomp}
$\Rext$ has the spectral decomposition:
\begin{align*}
\ \Rext \ =\  \Qext \Lambdaext \Qext ,
\end{align*}
where $\Lambdaext$ is the diagonal matrix containing the eigenvalues of
$\Rext$,
which can be obtained by $\sqrt{s}$ times the Fourier transform on the
first column of $\Rext$,
and $\Qext = \Re (\cF) + \Im (\cF)$ is real symmetric,  with
$$
    \cF_{\bsk, \bsk'}
    =
    \frac{1}{\sqrt{s}} \exp \left({2 \pi} \ri
      \frac{\bsk' \cdot \bsk}{{2 m}} \right)
$$
denoting the {$d$-dimensional}  Fourier
matrix. If the eigenvalues of $\Rext$ are all non-negative then the
required $B$ in \eqref{eq:Rfactor} can be obtained by selecting $M$
appropriate rows of
\begin{equation*}
 \Bext := \Qext (\Lambdaext)^{1/2} .
\end{equation*}
\end{theorem}

The use of FFT allows fast computation of the matrix-vector product
$\Bext\bsy$ for any vector $\bsy$, which then yields $B\bsy$ needed for
sampling the random field in \eqref{eq:Gauss_exp1}. Our algorithm from
\cite{paper1} for obtaining a minimal positive definite $\Rext$ is given
in Algorithm~\ref{alg1}. Our algorithm from \cite{paper1} for sampling an
instance of the lognormal random field is given in Algorithm~\ref{alg2}.
Note that the normalisation used within the FFT routine
differs among particular implementations. Here,
we assume the Fourier transform to be unitary.

\begin{algorithm} \label{alg1}
Input: $d$, $m_0$, and covariance function $\rho$.
\begin{enumerate}[noitemsep,topsep=0pt]
\item Set $m = m_0$.
\item Calculate $\bsr$, the first column of $\Rext$ in
    \eqref{eq:Rext_def}.
\item Calculate $\bsv$, the vector of eigenvalues of $\Rext$, by
    $d$-dimensional FFT on $\bsr$.
\item If the smallest eigenvalue $<0$ then increment $m$ and go to Step~2.
\end{enumerate}
Output: $m$, $\bsv$.
\end{algorithm}

\begin{algorithm} \label{alg2}
Input: $d$, $m_0$, mean field $\overline{Z}$, and
$m$ and $\bsv$ obtained by Algorithm~\ref{alg1}.
\begin{enumerate}[noitemsep,topsep=0pt]
\item With $s = (2m)^d$, sample an $s$-dimensional normal random vector $\bsy$.
\item Update $\bsy$ by elementwise multiplication with $\sqrt{\bsv}$.
\item Set $\bsw$ to be the $d$-dimensional FFT of $\bsy$.
\item Update $\bsw$ by adding its real and imaginary parts.
\item Obtain $\bsz$ by extracting the appropriate $M=(m_0+1)^d$
    entries of $\bsw$.
\item Update $\bsz$ by adding $\overline{Z}$.
\end{enumerate}
Output: $\exp(\bsz)$.
\end{algorithm}

In the case of QMC sampling, the random sample $\bsy$ in Step~1 of
Algorithm~\ref{alg2} is replaced with a randomly shifted lattice
point from $[0,1]^s$, mapped to $\bbR^s$ elementwise by the
inverse of the cumulative normal distribution function (see
\eqref{eq:QMC}). The relative size of the quantities $b_j = \|\bsB_j\|_\infty$
(as defined in \eqref{eq:bj}) determines the ordering of the QMC
variables in order to benefit from the good properties
of lattice rules in earlier coordinate directions in the construction
of the generating vector in Remark~\ref{rem:CBC}.

We prove in \cite{paper1} (under mild conditions) that
Algorithm~\ref{alg1} will always terminate. Moreover, in many cases the
required $m$ (equivalently $\ell$) can be quite small. Theorem
\ref{cor:matern-growth} below gives an explicit lower bound for the
required value of $\ell$.

\begin{example} \label{ex:Mat}
The \Mat family of covariances are defined by
\begin{equation} \label{defmatern}
 \rho(\bsx) = \
 \sigma^2 \, \frac{2^{1-\nu}}{\Gamma(\nu)}
 \bigg(\frac{\sqrt{2\nu}}{\lambda}\,  \|\bsx\|_2\bigg)^{\nu}
 K_\nu\bigg( \frac{\sqrt{2\nu}}{\lambda}\,  \|\bsx\|_2\bigg)\,,
\end{equation}
where $\Gamma$ is the Gamma function and $K_\nu$ is the modified Bessel
function of the second kind, $\sigma^2$ is the variance,  $\lambda$ is the
correlation length and $\nu \geq  1/2$ is a smoothness
parameter. The limiting cases $\nu \to 1/2$ and $\nu \to \infty$
correspond to the exponential and Gaussian covariances respectively,
see, e.g., \cite{GKNSSS:15}, however, using a slightly different scaling.
\end{example}

The following result, proved in \cite[Thm.~2.10]{paper1}, shows that
the growth of the size of $\ell$ with respect to the mesh size $h_0$ and
with respect to the parameters in the \Mat family is moderate. In
particular, for fixed  $\nu< \infty$, it establishes a bound on $\ell$
that grows only logarithmically with $\lambda/h_0$ and gets smaller as
$\lambda$ decreases.  Experiments illustrating the sharpness of this bound
are given in \cite{paper1}.

\begin{theorem}\label{cor:matern-growth}
Consider the \Mat covariance family \eqref{defmatern} with $1/2 \leq \nu <
\infty$ and $\lambda \leq 1$. Suppose $h_0/\lambda \leq e^{-1}$. Then
there exist constants $C_1>0$ and $C_2\geq 2 \sqrt{2}$ which may depend on
$s$ but are independent of $\ell, h_0, \lambda, \nu$ and $\sigma^2$, such
that $\Rext$ is positive definite if
\begin{align*}
\ell/\lambda  \ \geq  \  C_1\  + \ C_2\,  \nu^{1/2} \,
\log\left( \max\{ {\lambda}/{h_0}, \, \nu^{1/2}\} \right) \ .
\end{align*}
In the case $\nu=\infty$, the bound on $\ell$ is of the form $\ell
  \ge 1 + \lambda \, \max\{ \sqrt{2}\lambda/h_0, C_1\}$.
\end{theorem}

In order to verify the QMC convergence estimate given in Theorem
\ref{thm:QMC} in the case of circulant embedding, we need to bound
$\Vert \bsb \Vert_{p,s}$, where $\bsb$ is
defined in \eqref{eq:bj}. Since every entry in  $\Re(\calF) +
\Im(\calF)$ is bounded by $\sqrt{2/s}$, we have
\begin{equation} \label{eq:bj-lam}
  b_j \,=\, \|\bsB_j\|_\infty \,\le\, \sqrt{\frac{2}{s}\,\Lambda^{\rm ext}_{s,j}}\,,
\end{equation}
where $\Lambda^{\rm ext}_{s,j}$, $j=1,\ldots,s$, are the eigenvalues of
the nested block circulant matrix $\Rext$. Notice that we added `$s$'
explicitly to the notation to stress the dependence of these eigenvalues
on $s$.
A sufficient condition to ensure the uniform boundedness of
$\|\bsb\|_{s,p}$ required in
Theorem~\ref{thm:QMC} is that there exists a constant $C>0$,
independent of $s$, such that
 \begin{equation} \label{eq:QMCcriterion}
 \sum_{j=1}^{s}  \left(\frac{\Lambda^{\rm ext}_{s,j}}{s} \right)^{p/2}
 \ \leq \ C \ .
\end{equation}
It is thus important to investigate for what values of $p$ this inequality
holds. The smaller the value of $p$ the faster the convergence will be in
Theorem~\ref{thm:QMC}.

In \cite[\S 3]{paper1}, we conjecture (with supporting mathematical
arguments and empirical evidence) that the eigenvalues $\Lambda^{\rm
ext}_{s,j}$, when rearranged in non-increasing order, decay like
$j^{-(1+2\nu/d)}$ in case of the \Mat covariance.
This is the same as the decay rates of both the eigenvalues of the
original nested block Toeplitz matrix $R$ and of the KL eigenvalues of
the underlying continuous field~$Z$.
Under this conjecture, it follows that the smallest value of $p$ allowed for
\eqref{eq:QMCcriterion} to hold is just bigger than $2/(1+2\nu/d)$.
In turn this yields a theoretical convergence rate of nearly
\[
  \calO(n^{-\min(\nu/d,1)})
\]
in Theorem~\ref{thm:QMC} above, for any $\nu > d/2$,
independently of $s$. To see this,
recall that the convergence rate of $-1/(2\kappa)$ with respect to $n$ in
Theorem~\ref{thm:QMC} is related to $p$ via $p=2\kappa/(1+\kappa)$
with $\kappa \in (1/2,1)$. These bounds on $\kappa$ imply that, for the
conjectured rate of decay of the eigenvalues $\Lambda^{\rm ext}_{s,j}$,
Theorem~\ref{thm:QMC} is only applicable for $\nu > d/2$.

These conjectures will be investigated in detail in our numerical
experiments in the next section. As we will see there, the theoretically
predicted rates may be pessimistic. In the experiments here, we see nearly
optimal QMC convergence, i.e., $\mathcal{O}(n^{-1})$, even when $\nu < d$, and
at least as good convergence as for standard MC, i.e., $\mathcal{O}(n^{-1/2})$,
even when $\nu < d/2$. All these findings are in line with the results we
obtained in the case of KL expansions in \cite{GKNSSS:15}, and they guarantee a
dimension-independent optimal QMC convergence for sufficiently large smoothness
parameter $\nu$.

\section{Numerical Experiments}
\label{sec:Numerical}

In this section we perform numerical experiments on problem
\eqref{eq:diffeq} in 2D and 3D which illustrate the power of the proposed
algorithm. Our quantity of interest will be the average value of the
solution $u$
\begin{equation}\label{eq:QoI}
 \cG(u(\cdot,\bsy))  \ = \  \frac{1}{\vert T \vert} \int_T u(\bsx,\bsy) \,
\rd \bsx \ ,
\end{equation}
over some measurable $T\subseteq D$, with $D$ being an $L$-shaped domain
with a hole in 2D or the unit cube in 3D; all details to be specified
below. In both cases, the domain $D$ is contained in the unit cube
$[0,1]^d$, as assumed.\medskip

\noindent{\bf Random field generation.}
In all experiments the random coefficient $a$ is of the form
\eqref{eq:field} where $Z$ is a Gaussian random field with the \Mat
covariance \eqref{defmatern}. We take the mean $\overline{Z}$ to be $0$,
the variance to be $\sigma^2 = 0.25$, and we consider two different values
for the correlation length, namely $\lambda \in\{0.2, 0.5\}$, combined
with three different values for the smoothness parameter $\nu \in \{0.5,
2, 4\}$ in 2D and $\nu \in\{0.5, 3, 4\}$ in 3D (thus illustrating the cases
$\nu<d$, $\nu=d$, $\nu>d$ in each case). The forcing term is taken to be
$f\equiv 1$.

For different values of $m_0$, we first obtain values of the random field
on a uniform grid with $(m_0+1)^d$ points on the unit cube $[0,1]^d$ by
circulant embedding as described in \S\ref{sec:expect}. We choose $m_0\in
\{12, 24, 48, 96\}$ in 2D and $m_0\in \{7, 14, 28\}$ in 3D. The necessary
length $\ell = m/m_0$ of the extended cube $[0,\ell]^d$ to ensure positive
definiteness, where $m\ge m_0$, depends on the values of $d$, $\lambda$
and $\nu$, and affects the dimensionality $s=(2m)^d$ of $\bsy$. This
dependence is investigated in detail in \cite{paper1} (see Theorem
\ref{cor:matern-growth}). In
Table~\ref{tab:dim}, we summarise the values of $s$ for the different
combinations of parameters.\medskip

\newcommand{\cf}[1]{\tiny(#1)}
\newcommand{\nbe}[1]{\tiny(nbe=#1)}
\begin{table} [t]
\footnotesize
\begin{center}
\begin{tabular}{|l|rrr|rrr|}
 \hline
 $d=2$ & \multicolumn{3}{c|}{$\lambda =0.2$} & \multicolumn{3}{c|}{$\lambda =0.5$}  \\
 & $\nu=0.5$ & $\nu=2$ & $\nu=4$  & $\nu=0.5$ & $\nu=2$ & $\nu=4$ \\
 \hline
 $m_0 = 12$ &  576 &  576 &  576 &  1,296 &   5,476 &   9,216 \\[-1.5mm]
 \nbe{424}  & \cf{0.07} & \cf{0.03} & \cf{0.04} & \cf{0.11} & \cf{0.22} & \cf{0.15} \\
 $m_0 = 24$ & 2,304 &  2,916 &  4,900 &  8,464 &  34,596 &  59,536 \\[-1.5mm]
 \nbe{1,255} & \cf{0.04} & \cf{0.04} & \cf{0.05} & \cf{0.09} & \cf{0.26} & \cf{0.39} \\
 $m_0 = 48$ & 9,216 & 19,044 & 33,124 & 49,284 & 198,916 & 350,464 \\[-1.5mm]
 \nbe{5,559} & \cf{0.03} & \cf{0.05} & \cf{0.07} & \cf{0.15} & \cf{0.32} & \cf{0.44} \\
 $m_0 = 96$ & 36,864 & 114,244 & 200,704 & 270,400 & 1,077,444 & 1,721,344 \\[-1.5mm]
 \nbe{23,202} & \cf{0.02} & \cf{0.06} & \cf{0.11} & \cf{0.14} & \cf{0.40} & \cf{0.49} \\
 \hline
 \multicolumn{7}{c}{} \\
 \hline
 $d=3$ & \multicolumn{3}{c|}{$\lambda =0.2$} & \multicolumn{3}{c|}{$\lambda =0.5$}  \\
 & $\nu=0.5$ & $\nu= 3$ & $\nu=4$  & $\nu=0.5$ & $\nu=3$ & $\nu=4$ \\
 \hline
 $m_0 = 7$  &   2,744 &   2,744 &   2,744 &  64,000        & 97,336         & 125,000 \\[-1.5mm]
 \nbe{2,642} & \cf{0.002} & \cf{0.001} & \cf{0.001} & \cf{0.01} & \cf{0.02} & \cf{0.03}   \\
 $m_0 = 14$ &  27,000 &  39,304 &  39,304 & 1,061,208       & 2,000,376  & 2,406,104 \\[-1.5mm]
 \nbe{21,491}& \cf{0.001} & \cf{0.002} & \cf{0.002} & \cf{0.04} & \cf{0.06} & \cf{0.10} \\
 $m_0 = 28$ & 438,976 & 778,688 & 941,192 & 15,625,000 & 30,371,328 & 37,933,056 \\[-1.5mm]
 \nbe{172,421}& \cf{0.002} & \cf{0.003} & \cf{0.004} & \cf{0.06} & \cf{0.12} & \cf{0.14} \\
 \hline
\end{tabular}
\end{center}
\caption{The values of $s$ needed to ensure positive definiteness for
different combinations of parameters in 2D and 3D. The numbers in round
brackets show the cost of random field generation as a fraction of the
total computational cost per sample. These numbers increase with
increasing $s$ (from left to right) for a fixed FE mesh. For reference,
we also provide the number of elements \texttt{nbe} for each of the FE
meshes.}
\label{tab:dim}
\end{table}

\noindent{\bf FE solution of the PDE.}
For each realisation of the random field (i.e., for each $\bsy$), we
solve the PDE using a finite element (FE) method with piecewise linear
elements on an unstructured grid produced with the help of the {\sc
Matlab} PDE toolbox. The quadrature rule for the matrix assembly is
based on the mid point rule, where the values of the random field at the
centroids of the elements are obtained by multi-linear interpolation of
the values on the uniform grid, computed with circulant embedding (see
\eqref{eq:defAh} and \eqref{eq:defatau}). In order to balance the
quadrature error and the FE discretisation error in light of
Lemma~\ref{lem:quad} and Theorem~\ref{thm:H1error}, the
maximum FE mesh diameter is chosen such that $h\approx \sqrt{d}\,h_0
= \sqrt{d}/m_0$. In particular, we choose $h\in\{0.12, 0.06, 0.03,
0.015\}$ in 2D and $h\in\{0.24, 0.12, 0.06\}$ in 3D, for each of the
respective values of $m_0$ above. In 2D,  the {\sc Matlab} function {\tt
adaptmesh} is used to build a family of adaptive meshes for the $L$-shaped
domain with a hole (see Fig.~\ref{fig:L-shape}). We use the same adaptive
mesh, constructed with $a \equiv 1$, for all realisations. To find meshes
with our desired maximum mesh diameters $h$, we gradually increase the
\texttt{maxt} parameter of the \textsc{Matlab} \texttt{adaptmesh} command.
Fig.~\ref{fig:interp} zooms in on the shaded region in the bottom left
corner of each of the adaptive meshes to show the centroids of the
triangles in relation to the uniform grids. The PDE is solved with the
{\sc Matlab} function  {\tt assempde}.
For the 3D problem, we use the {\sc Matlab} PDE toolbox to mesh and solve
the PDE. The integral in \eqref{eq:QoI} is approximated by applying the
midpoint rule on each of the elements in $T$. In 2D, the resulting
linear system is solved with the default sparse direct solver
(``backslash'') in {\sc Matlab}. We believe that that is also the
solver used in the {\sc Matlab} PDE toolbox for our 3D experiments, but
we could not verify this.
\begin{figure}[t]
 \centering
 \includegraphics[width=0.3\textwidth]{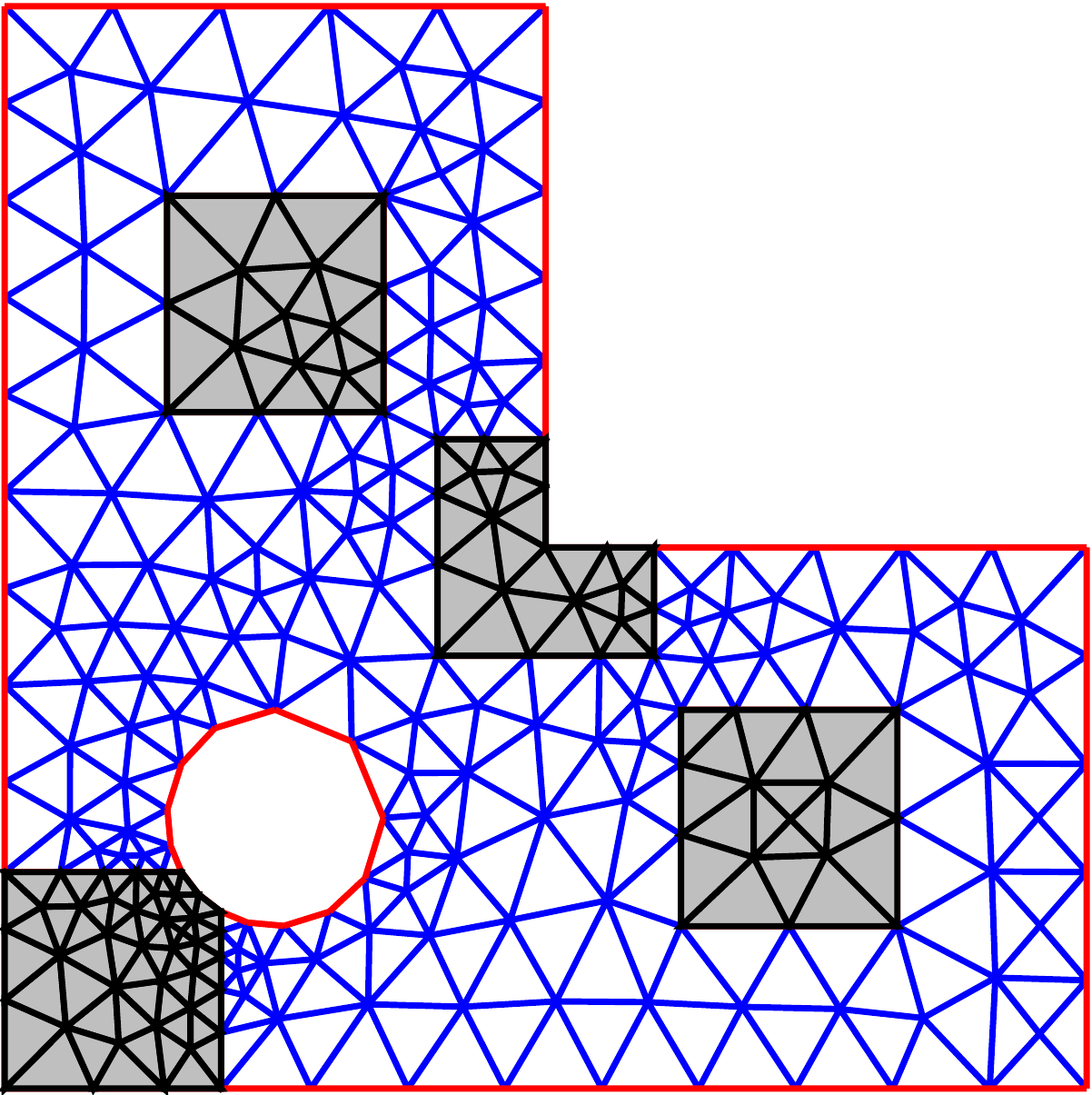} \quad
 \includegraphics[width=0.3\textwidth]{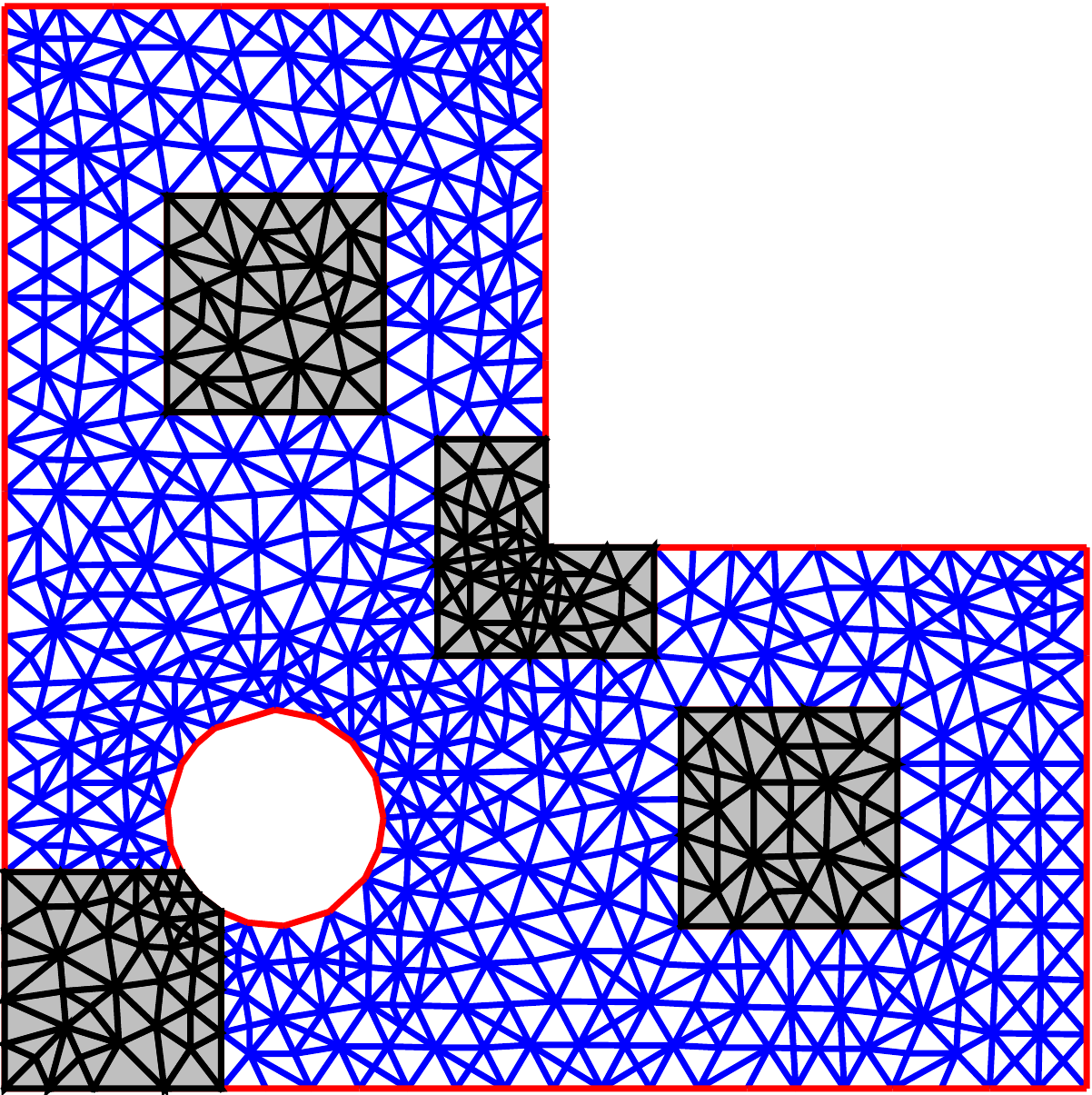} \quad
 \includegraphics[width=0.3\textwidth]{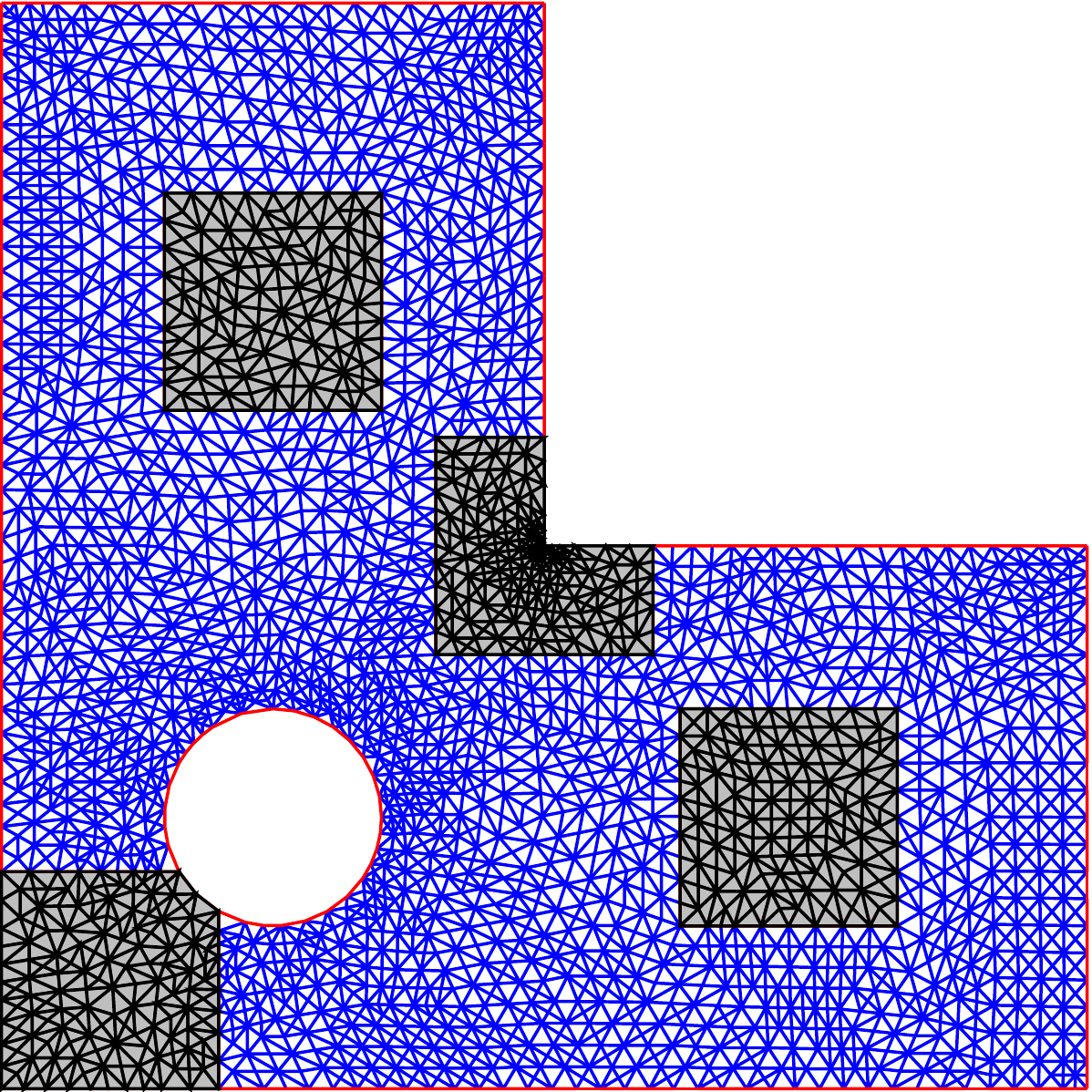}
 \caption{Adaptive FE mesh of an $L$-shaped domain with a hole.
 Left: $h=0.12$ (424 elements). Middle: $h=0.06$ (1,255 elements).
 Right: $h=0.03$ (5,559 elements).  Our fourth FE mesh not shown here:
 $h = 0.015$ (23,202 elements).
}\label{fig:L-shape}
\end{figure}
\begin{figure}[t]
\centering
 \includegraphics[width=0.3\textwidth]{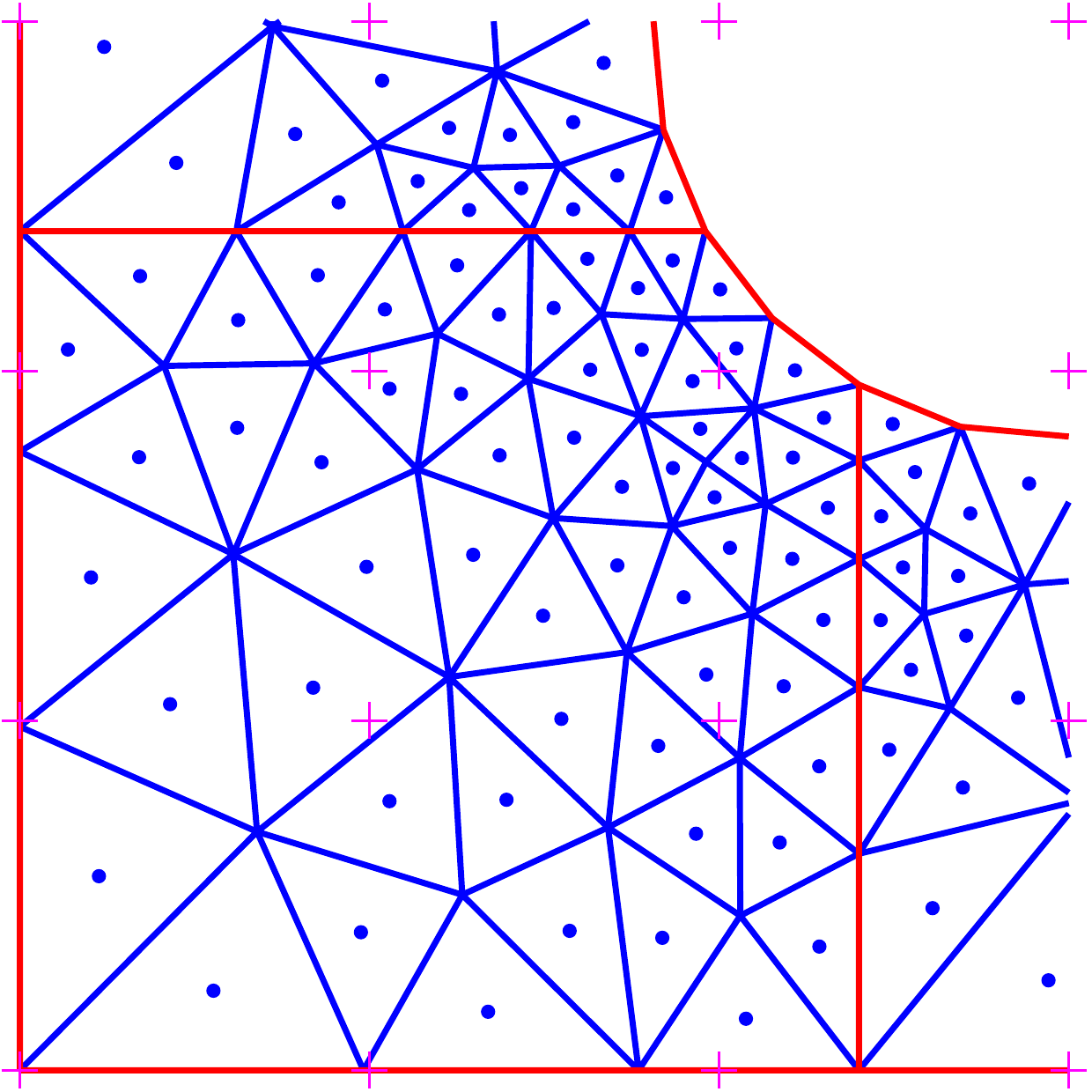} \quad
 \includegraphics[width=0.3\textwidth]{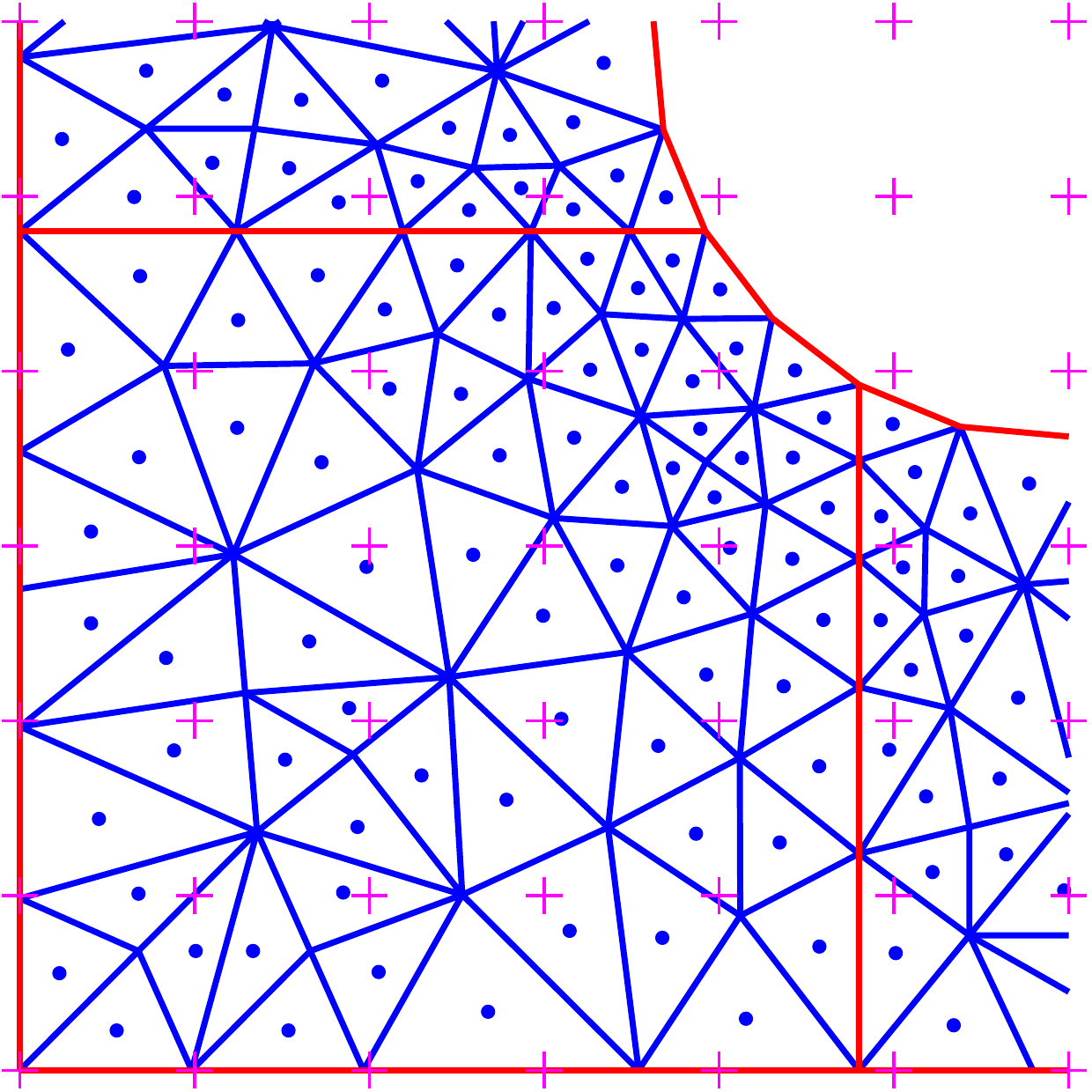} \quad
 \includegraphics[width=0.3\textwidth]{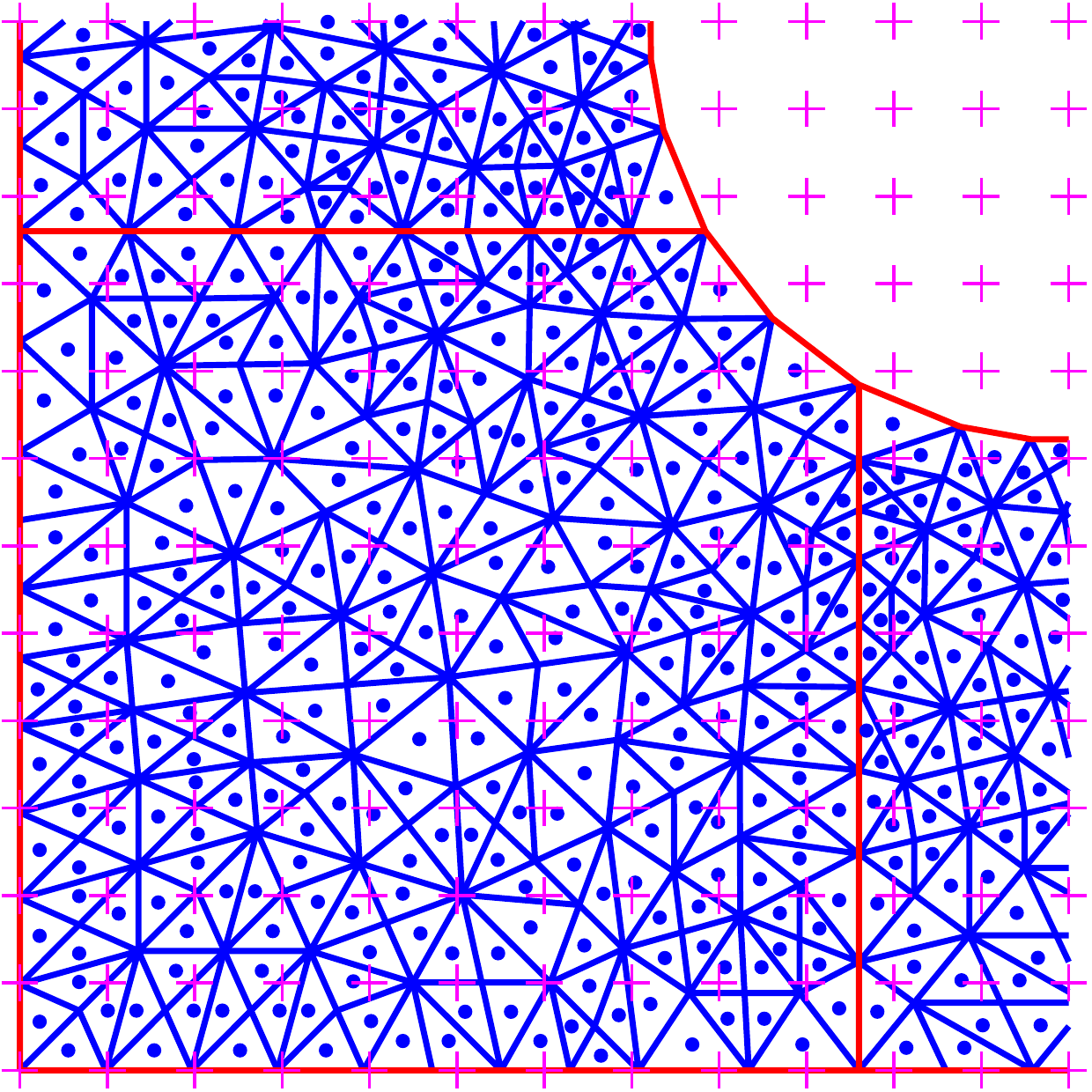}
 \caption{A local view of the meshes from
   Fig.~\ref{fig:L-shape}, showing the quadrature points at the
 centroids of the triangles (blue dots) and the
 uniform grid points where the random field is sampled (purple crosses).
 Left: $(m_0,h)=(12,0.12)$. Middle: $(m_0,h)=(24,0.06)$. Right:
 $(m_0,h)=(48,0.03)$.
}\label{fig:interp}
\end{figure}

As we can see from the fraction of time needed to construct the random
field, which is shown in brackets in Table~\ref{tab:dim} for each case,
the majority of time is spent on assembling and solving the FE systems. As
expected this is even more pronounced in 3D, since sparse direct
solvers are known to be significantly more expensive for 3D FE
matrices with respect to the number of unknowns (both in terms of the
order and in terms of the asymptotic constant), while the cost
of the FFT factorisation grows log-linearly with the number of
unknowns in 2D and in 3D.
In any case, the random field generation is
insignificant in the majority of cases and it takes less than 50\% of
the computational time in all cases.\medskip

\noindent{\bf Construction of lattice sequences.}
We approximate the expected value of \eqref{eq:QoI} by randomly shifted
lattice rules obtained using the fast CBC code from the QMC4PDE website
\url{https://people.cs.kuleuven.be/~dirk.nuyens/qmc4pde/} (see also
\cite{KuNu:15}). A typical call to the QMC4PDE construction script would
be:
\begin{center}
\small \texttt{./lat-cbc.py --s=2000 --b\_file=[f] --m=16 --d1=1 --d2=2
--a3=1 --outputdir=[o]}
\end{center}
where \texttt{s=2000} specifies an initial maximum number of dimensions and
\texttt{[f]} is a file containing the calculated values of $b_j$ in
\eqref{eq:bj} in nonincreasing order for a particular case of $d$,
$\lambda$ and~$\nu$. See \cite{KuNu:15} for an explanation of the
other parameters.

In specifying parameters for the CBC construction, we follow the theory of
\cite{KuNu:15}  as closely as possible, but we make a couple of
modifications for practical reasons. Firstly, the implementation
follows \cite{CKN06} to construct ``embedded'' lattice sequences in
base~$2$, so that in practice we can increase $n$ gradually without
throwing away existing function evaluations. At every power of~$2$, the
points form a full lattice rule which complies with the theory from
\S\ref{sec:qmc}. Secondly, with the POD weights in \eqref{eq:gamma} the
CBC construction to find the generating vector $\bsz$ has a cost of
$\calO(s^2n + sn\log
n)$ operations which becomes quite high for the large $s$ we are
considering (see Table~\ref{tab:dim}). Thus, we only carry out the CBC
construction up to a certain dimension $s^*$ and then randomly generate
the remaining components of $\bsz$. In particular, we
stop the CBC algorithm at the first component $s^*$ where the generating
vector has a repeated entry. Repeated components in $\bsz$
yield bad two-dimensional projections of lattice points and randomly
generated components are intuitively better in that situation. The
highest dimensionality for the switch-over dimension for all cases in
Table~\ref{tab:dim} is $s^*=1811$. The only two cases where we did not
need to add random components were $\nu=2$ and $\nu=4$, for $d=2$,
$m_0=12$ and $\lambda = 0.2$.\medskip

\noindent{\bf Estimation of (Q)MC error.}
For fixed $h$, we can compute the standard error on the QMC estimate
of the expected value of \eqref{eq:QoI} by using a number of
random shifts. Specifically, for each case we took $q=64$
independent random shifts of one $n$-point lattice rule, giving $q$
independent approximations $Q_1,\ldots,Q_q$ to the expected value. We take
their average $\overline{Q} = (Q_1+\cdots+Q_q)/q$ as our final
approximation, and we estimate the standard error on $\overline{Q}$ by
$\sqrt{\frac1{q(q-1)}\sum_{i=1}^q (Q_i-\overline{Q})^2}$.
The total number of function evaluations in this case is $N = q\,n$.
According to our theory (see also \cite{KuNu:15}), the
convergence rate for our randomised QMC
method is of the order $q^{-1/2}\,n^{-r} = q^{r-1/2}\,N^{-r}$, with
$r\approx\min(\nu/d,1)$. Hence, for $r > 1/2$, the constant in
any of the convergence graphs with respect to $N$ depends on $q^{r-1/2}$.
To provide less erratic curves, the number of random shifts is chosen to
be fairly large here. In practice, e.g., $q=16$ shifts would be
sufficient. This would effectively push all convergence graphs down,
leading to bigger gains for QMC.

We compare QMC with the simple Monte Carlo (MC)
method~\eqref{eq:MC} based on $N$ random samples. Denoting the function values
for these samples by $Y_1,\ldots Y_N$, then the MC approximation
of the integral is $\overline{Y} = (Y_1+\cdots + Y_N)/N$. The standard
error can be estimated by $\sqrt{\frac{1}{N(N-1)} \sum_{i=1}^N
(Y_i-\overline{Y})^2}$.  The expected MC convergence rate is
$\mathcal{O}(N^{-1/2})$.

In our figures later, we plot the \emph{relative standard error}
obtained by dividing the estimated standard error by the estimated
mean.\medskip

\noindent{\bf Computing environment.}
All our computations were run as serial jobs on reserved 8-core Intel
Xeon E5-2650v2 \@ 2.60 GHz nodes on the computational
cluster Katana at UNSW (or on almost identical hardware).
Since they are embarrassingly parallel, both the MC and QMC
simulations could easily be parallelised with roughly linear
speedup. We chose to run different jobs in parallel instead of
parallelising individual jobs, and to report the actual
serial computation times for our test cases.

\subsection{Results for an $L$-shaped domain in 2D}

In this example the domain $D$ is the complex 2D domain shown in
Fig.~\ref{fig:L-shape}: an $L$-shaped region with a hole. We consider
five choices for the averaging domain $T \subseteq D$ in
\eqref{eq:QoI}: \vspace{-0.2cm}
\begin{itemize}
\item [\textbf{T1}] the full domain, \vspace{-0.2cm}
\item [\textbf{T2}] the bottom left corner with a circular segment
    cut out,
    \vspace{-0.2cm}
\item [\textbf{T3}] the lower right interior square, and
    \vspace{-0.2cm}
\item [\textbf{T4}] the upper left interior square in a symmetrical
    location to \textbf{T3}, \vspace{-0.2cm}
\item [\textbf{T5}] the $L$-shape near the reentrant corner.
\end{itemize}

Fig.~\ref{fig:L-shape} shows the different averaging domains $T$, as
well as some of the adaptive meshes that were used. Note that
the circular sections of the boundary of $D$ are approximated
polygonally and the averaging domains
\textbf{T2}, \ldots, \textbf{T5} are resolved on all meshes. The meshes
are adapted to capture the loss of regularity near the reentrant corner,
but nevertheless the number of elements grows roughly with
$\mathcal{O}(h^{-2})$, the same as for a uniform family of meshes with
mesh size $h$. We specify the domain in
{\sc Matlab} by means of constructive solid geometry (CSG), i.e., the
union and subtraction of the basic pieces. This ensures that all the
averaging domains $T$ are covered by complete elements.\medskip

\noindent{\bf Mesh errors.}
Before we compare the performance of our QMC method with the basic MC
method, let us first estimate the discretisation errors for each of the
adaptive meshes. In Table~\ref{tab:2D-extrapolation}, we present results
for \textbf{T1} and \textbf{T5} for the case $\lambda=0.2$ and $\nu=2$.
The estimates of $\bE[\cG(u_h)]$, obtained using QMC, are stated
together with the estimated standard error for each mesh. We use
sufficiently many QMC cubature points, so that the significant
figures of the estimates are not affected
by QMC errors. The product of $m_0$ and $h$ is kept fixed
(approximately equal to $\sqrt{d}$), as discussed above. From the results
we can clearly see that the convergence rate for
the discretisation error is $\mathcal{O}(h^2)$ on the given
meshes. This shows that the mesh refinement near the reentrant
  corner is working optimally. Using
Richardson extrapolation, we can thus compute a higher order approximation
of the limit of $\bE[\cG(u_h)]$ that is stated in the last
row of Table~\ref{tab:2D-extrapolation}. The relative error with respect
to this extrapolated value is stated in the last column for both
\textbf{T1} and \textbf{T5}.
\begin{table}[t]
  \centering
 \begin{tabular}{|l|cc|cc|} \hline
 & \multicolumn{2}{|c|}{\textbf{T1}} &
                                       \multicolumn{2}{|c|}{\textbf{T5}} \\
    $(m_0,h)$ & $\bE[\cG(u_h)]$ & rel.~$h$-error &
                                                                     $\bE[\cG(u_h)]$
                                & rel.~$h$-error \\ \hline
    $(12,0.12)$ & 0.0114378 $\pm$ 3.4e-08\quad&  4.6e-02 & 0.0108212 $\pm$ 7.9e-08\quad & 7.9e-02 \\
    $(24,0.06)$ & 0.0118095 $\pm$ 2.8e-08\quad & 1.5e-02 & 0.0114086 $\pm$ 6.6e-08\quad & 2.9e-02 \\
    $(48,0.03)$ & 0.0119549 $\pm$ 4.1e-08\quad & 3.3e-03 & 0.0116850 $\pm$ 8.1e-08\quad & 5.9e-03 \\
    $(96,0.015)$ & 0.0119846 $\pm$ 4.9e-08\quad & 8.3e-04 & 0.0117372 $\pm$ 1.1e-07\quad & 1.5e-03\\ \hline
    & 0.0119945 {\footnotesize (extrapolated)}\hspace*{-0.5cm}& $\sim
                                                               h^2$ &
                                                                      0.0117546 {\footnotesize  (extrapolated)}\hspace*{-0.5cm}& $\sim h^2$ \\ \hline
  \end{tabular}
   \caption{Mesh convergence for the case $\lambda=0.2$ and $\nu=2$
     for \textbf{T1} and \textbf{T5}. The estimates of $\bE[\cG(u_h)]$
     (stated together with one standard deviation) are computed with our
     randomised lattice rule with $n$ sufficiently large, such that the
     standard error is significantly smaller than the discretisation
     error. Richardson extrapolation is used to compute a more
     accurate estimate of the limit of $\bE[\cG(u_h)]$, as $h\to
     0$ (final row). The columns denoted ``rel.~$h$-error'' give the relative error
     with respect to these extrapolated estimates.}\label{tab:2D-extrapolation}
\end{table}

The behaviour
is similar for the other quantities of interest and for the other values of
$\lambda$ when $\nu=2$. For $\nu=0.5$, on the other hand,
the solution is globally only in
$H^{3/2}$, so that the local mesh refinement near the reentrant corner
plays no role and the convergence of $\bE[\cG(u_h)]$ is only
$\mathcal{O}(h)$. Thus, to achieve
acceptable accuracy, comparable to the QMC errors we quote
below, we would in practice need much finer meshes for $\nu=0.5$.
However, since this would not
affect the behaviour of the QMC cubature errors, we did not do that.\medskip

\noindent
{\bf QMC convergence rates.}
\begin{figure}
  \centering
  \includegraphics{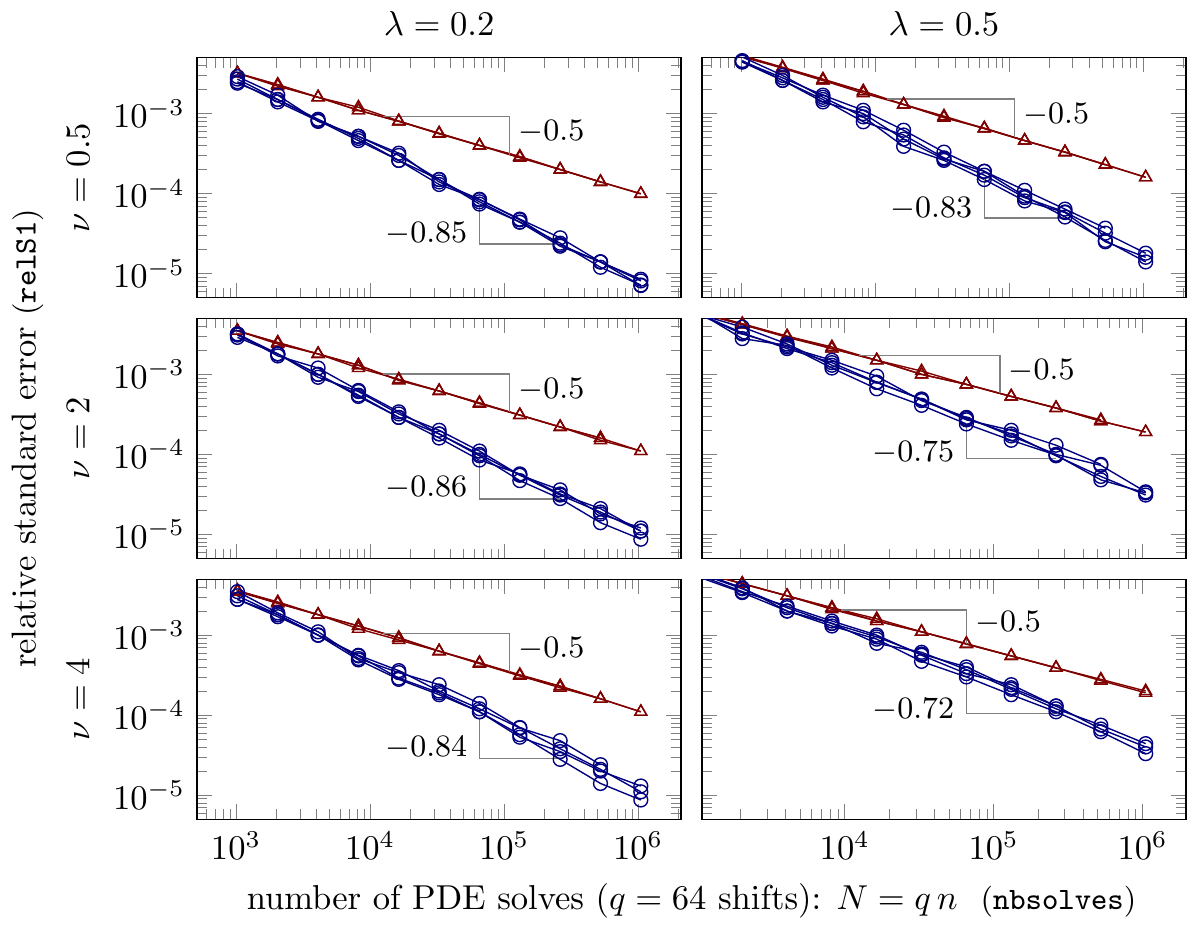} \\[4mm]
  \includegraphics{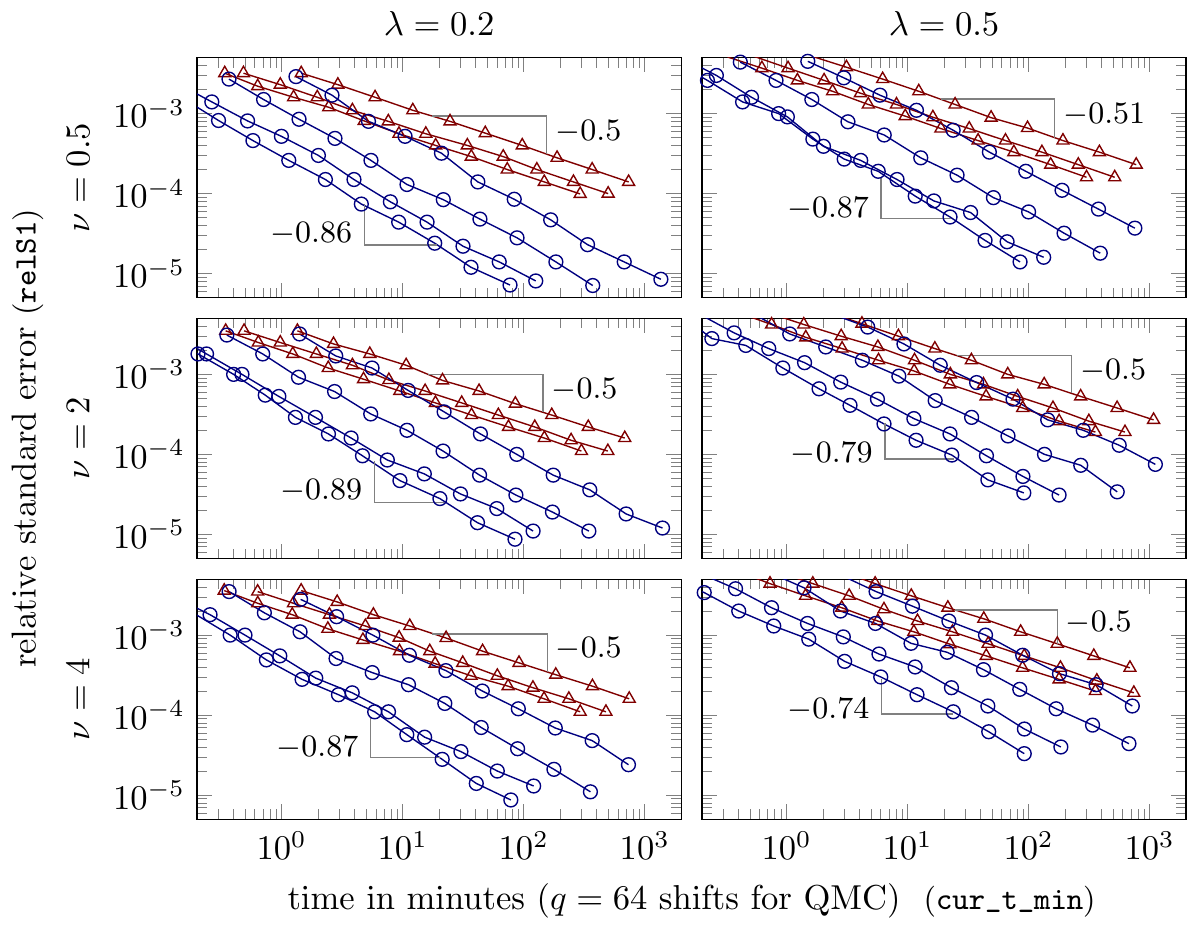}
  \caption{Relative standard error for $T=$ \textbf{T1}
(the average solution over the entire domain) against total number of
PDE solves (top) and execution time (bottom)
for Monte Carlo (red triangles) and QMC (blue
circles).
In the timing plot, results appear from bottom to top from the
coarsest mesh with $(m_0,h)=(12,0.12)$ to the finest mesh with
$(m_0,h)=(96,0.015)$ for QMC and $(m_0,h)=(48,0.03)$ for MC,
respectively.}\label{fig:2D-relS1}
\end{figure}
For the remainder we will now demonstrate the dimension independence of our QMC
method and its superiority over basic MC. In Fig.~\ref{fig:2D-relS1}, we
plot the relative standard error in~\eqref{eq:QoI} with $T=$ \textbf{T1} against
the total number of PDE solves (top) and against the total calculation
time (bottom). We consider six
combinations of the values of $\nu$ and
$\lambda$ and plot the graphs for four levels of mesh refinement. (The MC
estimates were not computed on the finest mesh.)
The convergence of the MC method is proportional to
$N^{-0.5}$, as expected. The convergence of the QMC method ranges from
$\mathcal{O}(N^{-0.72})$ up to $\mathcal{O}(N^{-0.89})$. For example, for
$\nu=0.5$ and $\lambda=0.2$, to achieve the same relative standard error of
$10^{-4}$ we need about $10^6$ PDE solves with the MC method while the QMC
method only needs about $3\cdot10^4$ PDE solves.
Also in terms of computational time, all the results consistently show huge
  computational savings for the QMC method over the MC method, even
  with the relatively large number of $q=64$ random shifts.

We note that the convergence graphs are meant to illustrate the
convergence behaviour and in practice one would not try to achieve such
high precision, especially not on the coarser meshes. One would rather aim
to balance the QMC errors with the discretisation errors in
Table~\ref{tab:2D-extrapolation}. From the theory, we expect the
smoothness $\nu$ of the random field to have an effect on the convergence
rate of QMC. Specifically, the bound in Theorem \ref{thm:QMC} would
suggest a rate of $\mathcal{O}(N^{-\min(\nu/2,1)})$ in 2D, as
discussed in \S\ref{sec:expect}. This effect is not immediately
observed in the graphs. For $\nu=2$ and~$4$ we expect an
asymptotic convergence rate of order~$1$, while we see very good
convergence rates on the graphs, we did not reach this asymptotic regime
yet. On the other hand, we do observe excellent convergence behaviour
for the case $\nu=0.5$ for which our theory does not apply.  From the
graphs we also observe that a smaller correlation length $\lambda$
corresponds to a better convergence rate, which is in full agreement
with our findings in \cite{GrKuNuScSl:11}.

A more important observation is the overlay of the convergence lines for the
different meshes in the plots of relative standard error versus
  number of PDE solves in Fig.~\ref{fig:2D-relS1}.
For example, for the case $\lambda=0.5$ and $\nu=4$ the
dimensionality $s$ increases from about 9 thousand to 1 million as we
increase $m_0$ from $12$ to $96$ (see Table~\ref{tab:dim}), while the
convergence rate and the asymptotic constant for the relative standard error
are clearly independent of the increasing dimension.

In Fig.~\ref{fig:2D-relS2-relS5}, we confirm that the superiority
  is independent of the quantity of interest, by presenting similar
  graphs as for \textbf{T1} also for \textbf{T2} and \textbf{T5}. We
  do not include the results for the two symmetrical squares
\textbf{T3} and \textbf{T4}, which look very similar.
\begin{figure}
  \centering
  \includegraphics{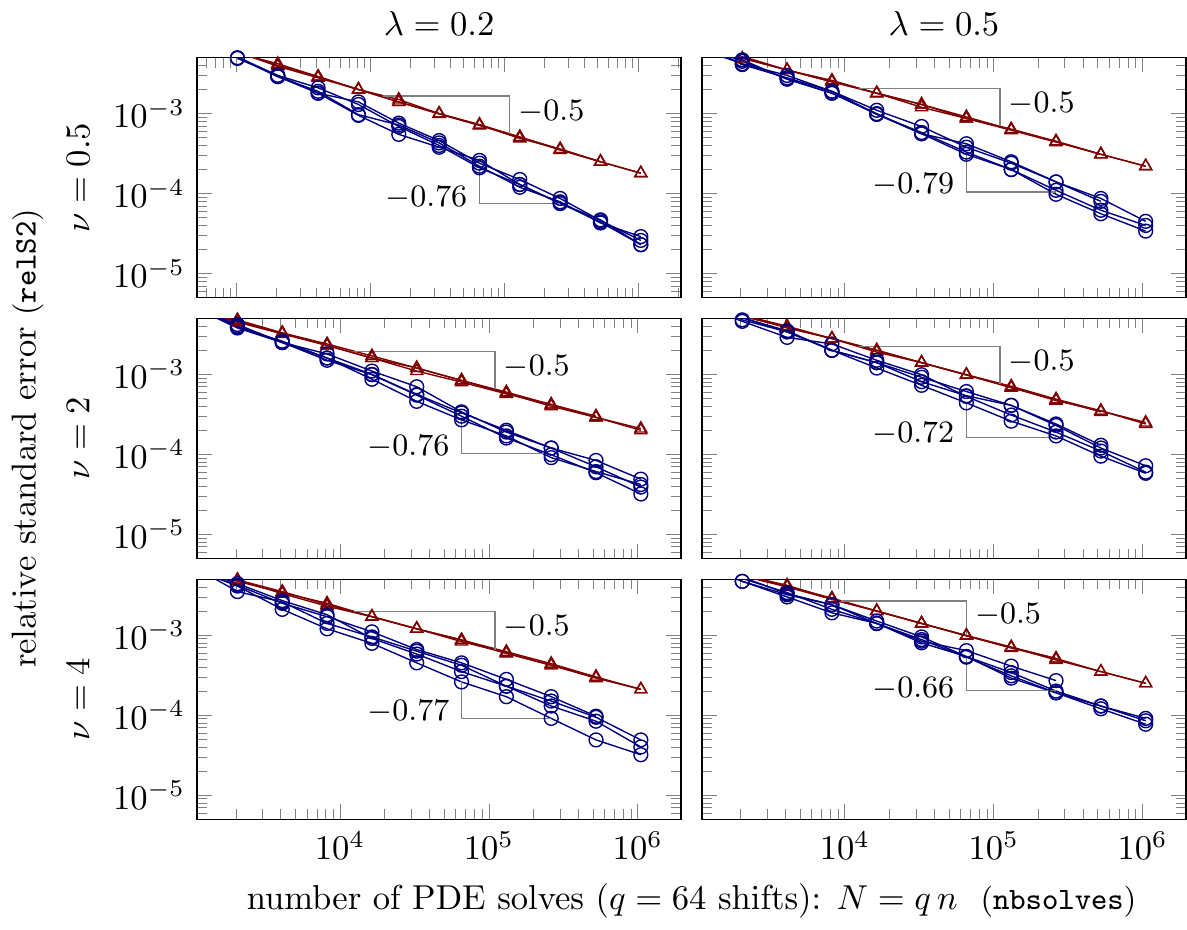} \\[4mm]
  \includegraphics{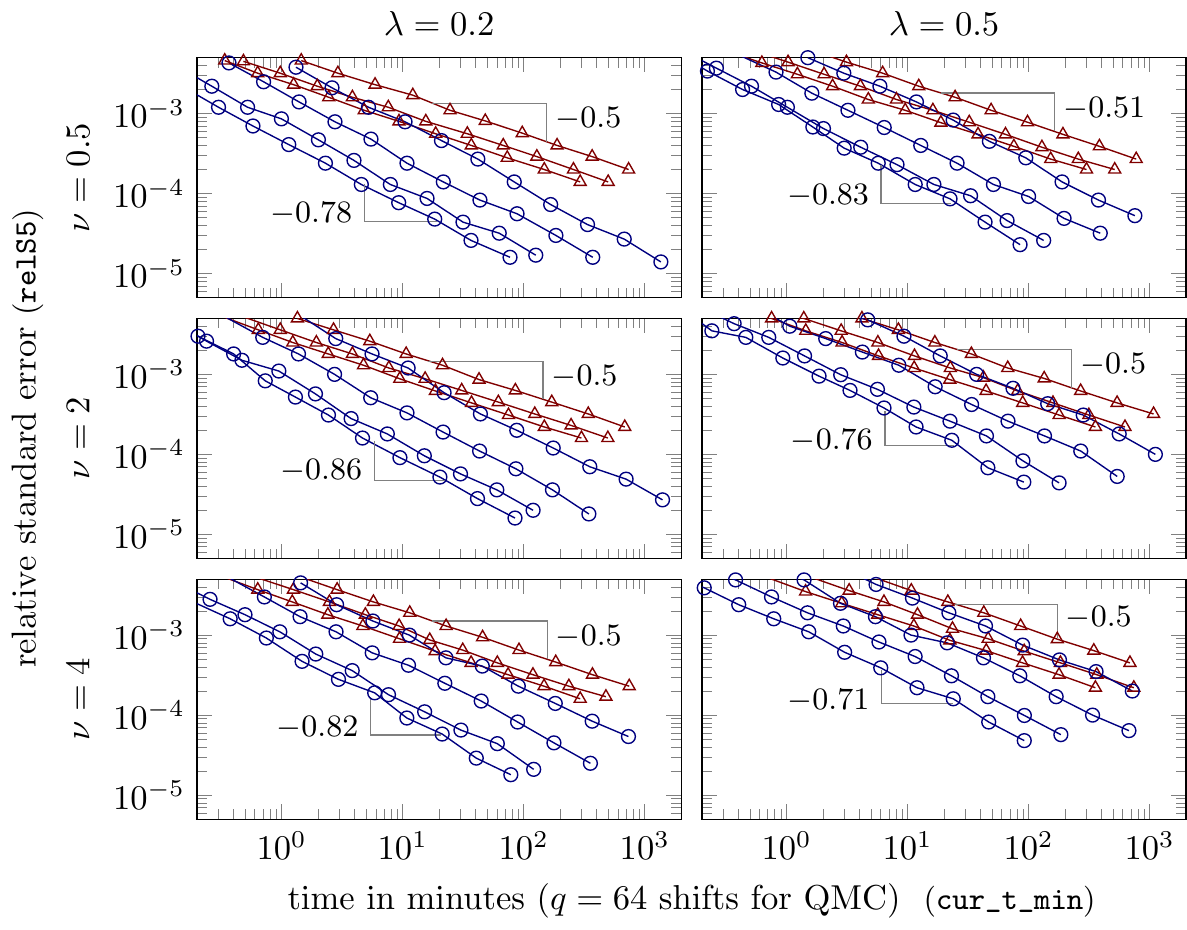}
  \caption{Relative standard error for Monte Carlo (red triangles)
      and the lattice sequences (blue circles) for \textbf{T2} (the average
      over the bottom-left corner with circular cutout) vs.\ number of
      PDE solves (top), as well as for \textbf{T5} (the average over
      the L-shaped region near the reentrant corner) vs.\ execution time (bottom).
  In the bottom figure, the results appear from bottom to top
  from the coarsest mesh to the finest.}\label{fig:2D-relS2-relS5}
\end{figure}

\subsection{Results for the unit cube in 3D}

Our second example is the unit cube in 3D and our quantity of interest is
\eqref{eq:QoI} with $T = D = [0,1]^3$. We use the mesh generator and
the solver from the \textsc{Matlab} PDE toolbox to obtain three meshes
with desired maximum mesh diameters $h$ and
to calculate the solution. To evaluate the random field $Z$ at the centroids
of the elements we used the function \texttt{interpolateSolution} from
the toolbox.
\begin{figure}
  \centering
  \includegraphics{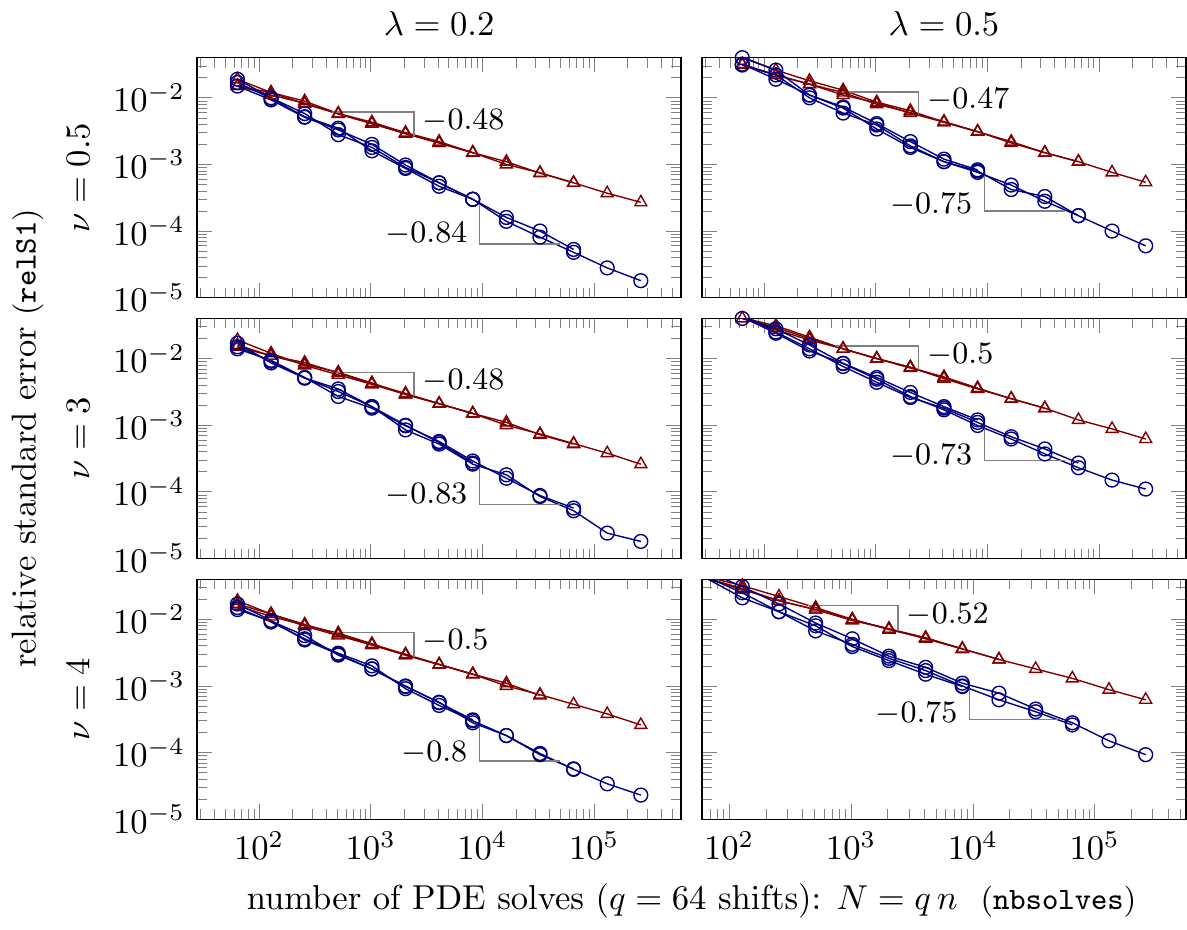}\\[4mm]
  \includegraphics{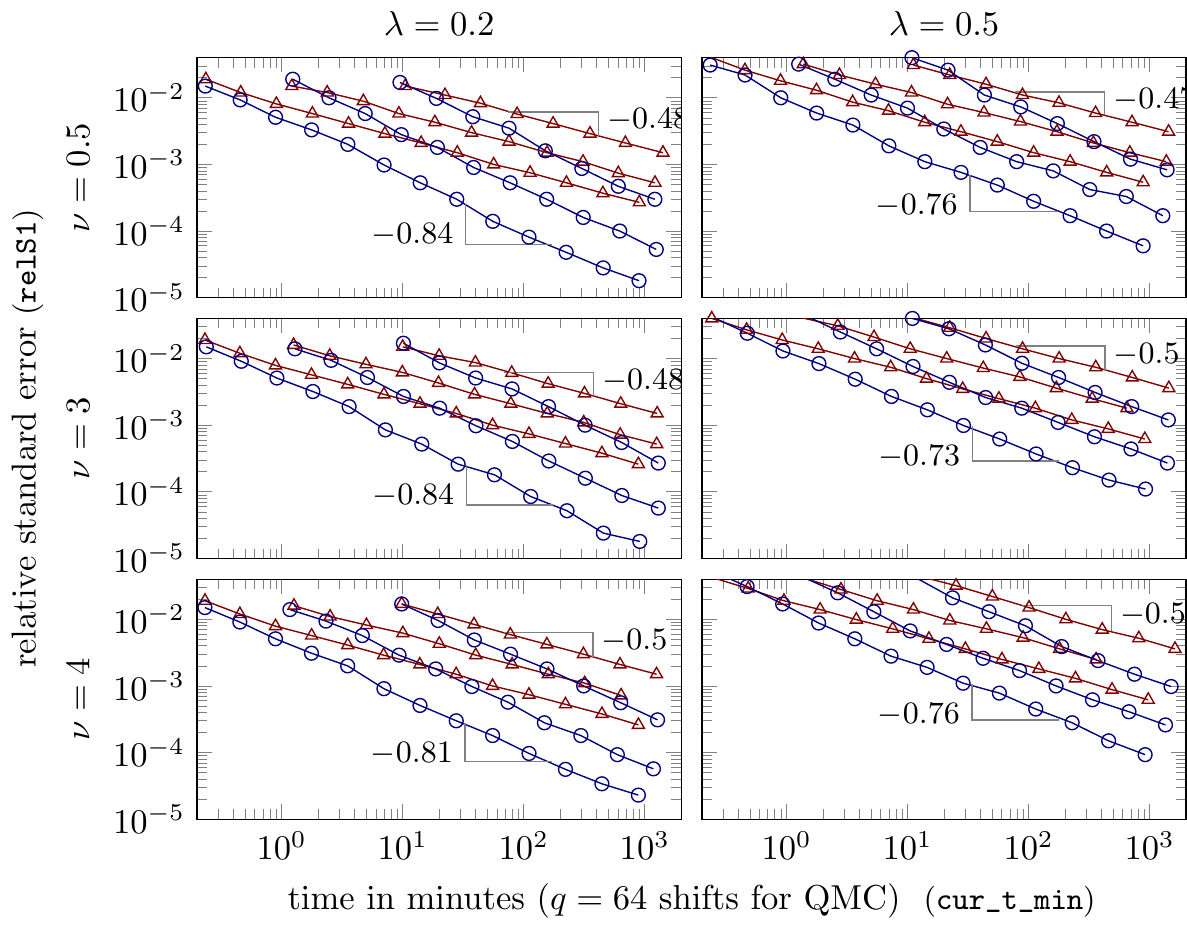}
  \caption{Relative standard error of the average solution over the
    unit cube in 3D against the total number of
 PDE solves (top) and the execution time (bottom)
for Monte Carlo (red triangles) and the lattice sequences (blue
circles). As in Fig.~\ref{fig:2D-relS1}, the results appear from
bottom to top from the coarsest mesh with
$(m_0,h)=(7,0.24)$ to the finest mesh with $(m_0,h)=(28,0.06)$.}
\label{fig:3D-relS1}
\end{figure}

In Fig.~\ref{fig:3D-relS1} we plot again the relative standard error
for six combinations of parameters against the total number of PDE
solves and against the calculation time. We observe convergence rates from
$N^{-0.73}$ up to $N^{-0.84}$ for our QMC rules, and
the expected $N^{-0.5}$ for the plain MC method. The lattice sequences
were again constructed with the actual values of $b_j$ from~\eqref{eq:bj},
except for the combination of $\lambda=0.5$ and $m_0 = 28$
where we get the near astronomical dimensionalities ranging from
about 15 to 37~million, as $\nu$ increases (see Table~\ref{tab:dim}); for
these cases we replaced $b_j$ by the convenient upper
bound given in~\eqref{eq:bj-lam}.\bigskip

\noindent {\bf Acknowledgement.} We gratefully acknowledge the financial
support from the Australian Research Council (FT130100655 and
DP150101770), the KU Leuven research fund (OT:3E130287 and C3:3E150478),
the Taiwanese National Center for Theoretical Sciences (NCTS) --
Mathematics Division, and the Statistical and Applied Mathematical
Sciences Institute (SAMSI) 2017 Year-long Program on Quasi-Monte Carlo and
High-Dimensional Sampling Methods for Applied Mathematics.


\begin{thebibliography}{99}

\bibitem{Ad:81}
R.J. Adler, \emph{The Geometry of Random Fields}, Wiley, London, 1981.


\bibitem{Apel:1999} T. Apel.
\newblock {\em Anisotropic Finite Elements: Local Estimates and Applications}.
\newblock Advances in Numerical Mathematics. Teubner, 1999.

\bibitem{BaNiZi:07} C.~Bacuta, V.~Nistor and L.~Zikatanov, Improving
  the rate of convergence of high-order finite elements on
  polyhedra {I}: {A} priori estimates, \emph{Numer.~Func.~Anal.~Opt.}
  \textbf{26}, 613--639, 2006.

\bibitem{BlCoDa:16}
P. Blanchard, O. Coulaud and E. Darve,
Fast hierarchical algorithms for generating Gaussian random fields,
\emph{Preprint hal-01228519}, 18 February 2016, available
at {\tt \verb+https://hal.inria.fr/hal-01228519+}

\bibitem{ChWo:94} G.~Chan and A.T.A.~Wood, {Simulation of stationary Gaussian
processes in $[0,1]^d$}, \emph{J. Comput. Graph. Stat.} \textbf{3}, 409-432,
1994.

\bibitem{ChWo:97} G. Chan and A.T.A. Wood, {Algorithm AS 312: An
  Algorithm for simulating stationary Gaussian random fields},
\emph{Appl. Stat.-J. Roy. St. C} \textbf{46}, 171--181, 1997.

\bibitem{ChScTe:11} J. Charrier, R. Scheichl and A. Teckentrup,
Finite element error analysis of elliptic PDEs with random
coefficients and its application to multilevel Monte Carlo methods,
{\em SIAM J. Numer. Anal.} {\bf 51},  322-352, 2013.

\bibitem{CKN06} R. Cools, F.Y. Kuo and D. Nuyens, Constructing embedded
   lattice rules for multivariate integration, \emph{SIAM J. Sci.
   Comput.} \textbf{28}, 2162--2188, 2006.

\bibitem{DiNe:97}
C.R. Dietrich and G.H. Newsam, {Fast and exact simulation of
stationary Gaussian processes through circulant embedding of the
covariance matrix}, \emph{SIAM J. Sci. Comput.} \textbf{18}, 1088--1107, 1997.

\bibitem{FeKuSl:17} M. Feischl, F.Y. Kuo and I.H. Sloan, Fast random
    field generation with $H$-matrices, \emph{Preprint
    arXiv:1702.0863728}, 28 February 2017, available at {\tt
    \verb+https://arxiv.org/abs/1702.08637+}

\bibitem{GKNSSS:15} I.G. Graham, F.Y. Kuo, J.A. Nichols, R. Scheichl, Ch.
    Schwab and I.H. Sloan, Quasi-Monte Carlo finite element methods for
    elliptic PDEs with log-normal random coefficients, {\em Numer.
    Math.} \textbf{131}, 329 -- 368, 2015.

\bibitem{GrKuNuScSl:11}
I.G. Graham, F.Y. Kuo, D. Nuyens,  R. Scheichl and I.H. Sloan,
Quasi-Monte carlo methods for elliptic PDEs with random coefficients and applications,
\emph{J. Comput. Phys.} \textbf{230}, 3668-3694, 2011.

\bibitem{paper1}
I.G. Graham, F.Y. Kuo, D. Nuyens,  R. Scheichl and I.H. Sloan,
 Analysis of circulant embedding methods for sampling stationary
 random fields,
\emph{Preprint arXiv:1710.00751}, 28 Sep 2017, available at {\tt
    \verb+https://arxiv.org/abs/1710.00751+}

\bibitem{GrScUl:13} I.G. Graham, R. Scheichl and E. Ullmann, Mixed finite
    element analysis of lognormal diffusion and multilevel Monte Carlo
    methods, {\em Stoch. PDE: Anal. Comp.} \textbf{4}, 41-75, 2013.

\bibitem{HaPeSc:12}
H. Harbrecht, M. Peters and R. Schneider,
On the low-rank approximation by the pivoted Cholesky decomposition,
\emph{Appl. Numer. Math.} \textbf{62}(4), 428-440, 2012.

\bibitem{KuNu:15} F.Y. Kuo and D. Nuyens, Application of quasi-Monte Carlo
    methods to PDEs with random coefficients - analysis and
    implementation, {\em Found.~Comput.~Math.}
    \textbf{16}, 1631–-1696, 2016.

\bibitem{LPS14}
  G. Lord, C. Powell and T. Shardlow,
  {\em An Introduction to Computational Stochastic PDEs}, Cambridge University
  Press, Cambridge, 2014.

\bibitem{NaHaSu:98a} R.L. Naff, D.F. Haley and E.A. Sudicky,
  {High-resolution Monte Carlo simulation of flow and conservative
  transport in heterogeneous porous media 1. Methodology and flow
  results}, \emph{Water Resour.~Res.}, \textbf{34}, 663--677, 1998.

\bibitem{NK14} J.A. Nichols and F.Y. Kuo.
\newblock Fast component-by-component construction of randomly shifted lattice
  rules achieving $\mathcal{O}(n^{-1+\delta})$ convergence for unbounded
  integrands in $\mathbb{R}^s$ in weighted spaces with {POD} weights.
\newblock {\em J.~Complexity}, \textbf{30}, 444--468, 2014.

\bibitem{ScWa:79} A.H. Schatz and L.B. Wahlbin, Maximum norm estimates
  in the finite element method on plane polygonal domains. {Part 2: R}efinements.
\emph{ Math. Comput.}  \textbf{33},  465--492,  1979.

\bibitem{TSGU13}
  A.L. Teckentrup, R. Scheichl, M.B. Giles and E. Ullmann,
  {Further analysis of multilevel Monte Carlo methods for elliptic
  PDEs with random coefficient},
  \emph{Numer. Math.} {\bf 125}, 569--600, 2013.

\end{thebibliography}
\end{document}